\documentclass[a4paper,10pt]{amsart}

\usepackage{amsmath,amsfonts,mathrsfs,amsthm,amssymb}
\usepackage[margin=2.5cm]{geometry}
\usepackage{stmaryrd}
\usepackage{tikz}
\usepackage{verbatim}
\usepackage[all]{xy}

\newtheorem{theorem}{Theorem}
\newtheorem{proposition}[theorem]{Proposition}

\newtheorem{corollary}[theorem]{Corollary}

\theoremstyle{definition}

\newtheorem{example}[theorem]{Example}
\newtheorem{remark}[theorem]{Remark}

\newcommand{\eps}{\epsilon}

\newcommand{\ul} [1]{\underline{#1}}
\newcommand{\wt} [1]{\smash{\widetilde{#1}}}

\newcommand{\bbC}{\mathbb{C}}

\newcommand{\bbI}{\mathbb{I}}
\newcommand{\bbJ}{\mathbb{J}}
\newcommand{\bbP}{\mathbb{P}}
\newcommand{\bbR}{\mathbb{R}}

\newcommand{\bbV}{\mathbb{V}}

\newcommand{\mbc}{\mathbf{c}}
\newcommand{\mbD}{\mathbf{D}}
\newcommand{\mbE}{\mathbf{E}}
\newcommand{\mbQ}{\mathbf{Q}}
\newcommand{\mbX}{\mathbf{X}}

\newcommand{\mcA}{\mathcal{A}}

\newcommand{\mcD}{\mbD}
\newcommand{\mcE}{\mathcal{E}}

\newcommand{\mcG}{\mathcal{G}}

\newcommand{\mcL}{\mathcal{L}}
\newcommand{\mcV}{\mathcal{V}}

\newcommand{\mfg}  {\mathfrak{g}}
\newcommand{\mfgl} {\mathfrak{gl}}

\newcommand{\mfso} {\mathfrak{so}}
\newcommand{\mfp}  {\mathfrak{p}}

\newcommand{\sfP}{\mathsf{P}}

\DeclareMathOperator{\csch} {csch}
\DeclareMathOperator{\End} {End}
\DeclareMathOperator{\G}   {G}
\DeclareMathOperator{\GL}  {GL}
\DeclareMathOperator{\gr}  {gr}
\DeclareMathOperator{\Hol} {Hol}
\DeclareMathOperator{\id}  {id}
\DeclareMathOperator{\im}  {im}
\DeclareMathOperator{\SL}  {SL}
\DeclareMathOperator{\SO}  {SO}
\DeclareMathOperator{\Spin}{Spin}
\DeclareMathOperator{\SU}  {SU}
\DeclareMathOperator{\Sym} {Sym}
\DeclareMathOperator{\tr}  {tr}

\newcommand{\fplusl}
{
\hspace{0.1cm}
\begin{tikzpicture}[baseline=-0.582ex]
    \draw [line width=0.24pt](-0.1129, 0) -- (0.1129, 0) -- (0, 0) -- (0, -0.1129) -- (0, 0.1129) arc (90:270:0.1129) -- (0, 0);
\end{tikzpicture}
\hspace{0.1cm}
}

\begin{document}

\title[The almost Einstein operator for $(2, 3, 5)$ distributions]{The almost Einstein operator\\ for $(2, 3, 5)$ distributions}
\author{Katja Sagerschnig and Travis Willse}

\address{
K.S.:
	Dipartimento di Scienze Matematiche\\ Politecnico di Torino\\ Corso Duca degli Abruzzi 24\\ 10129 Torino\\ ITALY  \\
T.W.:
	Fakult\"{a}t f\"{u}r Mathematik\\
	Universit\"{a}t Wien\\
	Oskar-Morgenstern-Platz 1\\
	1090 Wien\\
	AUSTRIA}

\email{katja.sagerschnig@univie.ac.at}
\email{travis.willse@univie.ac.at}

\subjclass[2010]{Primary 53A30, 53A40, 53C25, 58A30, 58J60}
\keywords{$(2, 3, 5)$-distributions, almost Einstein, BGG operators, conformal geometry, invariant differential operators}

\maketitle

\begin{abstract}
For the geometry of oriented $(2, 3, 5)$ distributions $(M, \mbD)$, which correspond to regular, normal parabolic geometries of type $(\G_2, P)$ for a particular parabolic subgroup $P < \G_2$, we develop the corresponding tractor calculus and use it to analyze the first BGG operator $\Theta_0$ associated to the $7$-dimensional irreducible representation of $\G_2$. We give an explicit formula for the normal connection on the corresponding tractor bundle and use it to derive explicit expressions for this operator. We also show that solutions of this operator are automatically normal, yielding a geometric interpretation of $\ker \Theta_0$: For any $(M, \mbD)$, this kernel consists precisely of the almost Einstein scales of the Nurowski conformal structure on $M$ that $\mbD$ determines.

We apply our formula for $\Theta_0$ (1) to recover efficiently some known solutions, (2) to construct a distribution with root type $[3, 1]$ with a nonzero solution, and (3) to show efficiently that the conformal holonomy of a particular $(2, 3, 5)$ conformal structure is equal to $\G_2$.
\end{abstract}

\tableofcontents

\section{Introduction}
A \textit{$(2, 3, 5)$ distribution} on a $5$-manifold $M$ is a $2$-plane distribution $\mbD\subset TM$  that is maximally nonintegrable in the sense that (1) $[\mbD, \mbD]$ is a $3$-plane distribution, and (2) $[\mbD, [\mbD, \mbD]] = TM$. There is an equivalence of categories between oriented $(2, 3, 5)$ distributions and regular, normal parabolic geometries of type $(\G_2, P),$ where $\G_2$ is the automorphism group of the algebra of split octonions and $P < \G_2$ is the maximal parabolic subgroup defined as the stabiliser of an isotropic ray in the $7$-dimensional fundamental representation $\mathbb{V}$ of $\G_2$ \cite[\S4.3.2]{CapSlovak}. (Here, \textit{oriented} just means that the bundle $\mbD \to M$ is oriented, or equivalently that $M$ itself is.) This correspondence is essentially established in Cartan's influential 1910 ``Five Variables Paper'' \cite{CartanFiveVariables}, where it is the output of Cartan's most involved application of his method of equivalence, and it has recently been exploited to give new insights into this geometry, often via the powerful general framework of parabolic geometry \cite{CapSagerschnig, HammerlSagerschnig, Nurowski, SagerschnigWillse}.

Besides this equivalence, for our purposes the most important feature of the geometry of $(2, 3, 5)$ distributions is its intimate connection to conformal geometry: Nurowski observed, implicitly by exploiting the inclusion $\G_2 \hookrightarrow \SO(3, 4)$, that any $(2, 3, 5)$ distribution $(M, \mbD)$ canonically induces a conformal structure $\mbc_{\mbD}$ of signature $(2, 3)$ on $M$ \cite{Nurowski}, a construction that has led to many results of independent interest over the last decade \cite{GPW, GrahamWillse, LeistnerNurowski}. We call a conformal structure that arises this way a \textit{$(2, 3, 5)$ conformal structure}, and in the oriented case these are characterised among all conformal structures of signature $(2, 3)$ precisely by the containment of the holonomy of the normal conformal connection in $\G_2$ \cite{HammerlSagerschnig, Nurowski}.

One of the motivations for this article is to establish an efficient calculus to determine whether a given $(2,3,5)$ conformal structure $(M, \mbc_{\mbD})$ contains an Einstein metric. If it does---or more precisely if it admits a nonzero almost Einstein scale (see below)---it inherits a rich system of additional structure and relationships with other geometric structures, as described in detail in \cite{SagerschnigWillse}.

For a general conformal structure $(M, \mbc)$ of signature $(p, q)$, $p + q \geq 3$, whether $\mbc$ contains an Einstein metric is governed by the conformally invariant, overdetermined, linear differential \textit{almost Einstein operator} \cite{BEG},
\begin{equation}\label{equation:BGG-conformal}
	\wt\Theta_0: \Gamma(\mcE[1]) \to {\textstyle \Gamma(\bigodot^2_{\circ} T^*M [1])},
		\qquad
	\wt\Theta_0(\sigma) := (\wt\nabla_a \wt\nabla_b \sigma + \wt\sfP_{ab}\sigma)_{\circ} .
\end{equation}
Here and henceforth, for a vector bundle $B \to M$, we denote $B[w] := B \otimes \mcE[w]$, where $\mcE[w] \to M$ is the bundle of conformal densities of weight $w$; $\wt\nabla$ denotes the Levi-Civita connection and $\wt\sfP$ the conformal Schouten tensor of a representative metric.   A nowhere-vanishing solution $\sigma \in \ker \wt\Theta_0$ determines an Einstein metric $\sigma^{-2}\mathbf{g}\in\mathbf{c}$, where $\smash{\mathbf{g}\in\Gamma(\bigodot^2 T^*M [2])}$ is the canonical weighted metric defined by the conformal structure $\mathbf{c}$, and vice versa. This motivates the term \textit{almost Einstein scale} for a solution $\sigma \in \Gamma(\mcE[1])$ (possibly with nonempty zero locus) of the differential equation $\wt\Theta_0(\sigma) = 0$. We say that $\mbc$ is itself \textit{almost Einstein} iff it admits a nonzero almost Einstein scale.

Writing $\wt\Theta_0(\sigma) = 0$ as a first-order system and prolonging once yields a closed system and hence determines a natural vector bundle $\mcV \to M$, called the \textit{conformal tractor bundle}, equipped with a conformally invariant \textit{normal tractor connection} $\nabla^{\mcV}$. In particular, the prolongation map $\wt L_0 : \Gamma(\mcE[1]) \to \Gamma(\mcV)$ and the projection $\wt\Pi_0 : \Gamma(\mcV) \to \Gamma(\mcE[1])$ onto the first component together define a bijective correspondence between the space $\ker \wt\Theta_0$ of almost Einstein scales and $\nabla^{\mcV}$-parallel sections of the bundle $\mcV$. If $(\wt\mcG, \wt\omega)$ is the normal conformal Cartan geometry of type $(\SO(p + 1, q + 1), \wt P)$ canonically associated to $(M, \mbc)$, then we can identify $\mcV$ with the associated bundle $\wt\mcG \times_{\wt P} \bbV$, where $\bbV \cong \bbR^{p + 1, q + 1}$ is the standard representation of $\SO(p + 1, q + 1)$, and under this identification $\nabla^{\mcV}$ is the connection induced thereon by $\wt\omega$ \cite{CapSlovak}; here, $\wt P < \SO(p + 1, q + 1)$ is the stabiliser of an isotropic ray in $\bbV$, and again it is a parabolic subgroup. This in turn realises $\wt\Theta_0$ as the \textit{first BGG operator} for conformal geometry corresponding to the representation $\bbV$ \cite{CalderbankDiemer, CSS}.

Given a $(2, 3, 5)$ conformal structure, we may identify the tractor bundle $\mcV$ associated to the standard representation $\bbV$ of $\G_2$ (which is just the restriction to $\G_2$ of the standard representation of $\SO(3, 4)$ of the same name) with that of the induced conformal structure, and the normal tractor connections $\nabla^{\mcV}$ coincide \cite[Proposition 4]{HammerlSagerschnig}. In particular, if $\sigma$ is an almost Einstein scale for the conformal structure, that is, $\sigma \in \mcE[1]$ is in the kernel of the almost Einstein operator $\wt\Theta_0$, then it is in the kernel of the BGG operator $\Theta_0$ for $(2, 3, 5)$ geometry associated to $\bbV$ (now regarded as a representation of $\G_2 \subset \SO(3, 4)$), which we compute explicitly in Section \ref{section:BGG-operator}. Conversely, if (and only if) a solution $\sigma \in \ker \Theta_0$ is \textit{normal}---that is, it is the projection $\Pi_0(t)$ of a $\nabla^{\mcV}$-parallel section $t \in \Gamma(\mcV)$---then it is an almost Einstein scale for the induced conformal structure $\mbc_{\mbD}$.

A priori it is conceivable that solutions of $\Theta_0$ need not always be normal, that is, that there are distributions for which the containment $\ker \smash{\wt\Theta_0} \subseteq \ker \Theta_0$ is proper: Indeed, we may view the conformal BGG equation $\wt\Theta_0(\sigma) = 0$ as an overdetermined system of $\dim \smash{\bigodot^2_{\circ} T^*M[1]} = 14$ partial differential equations in $\sigma$. By contrast, applying Kostant's Theorem \cite{CapSlovak} shows that $\Theta_0$ is an operator $\smash{\Gamma(\mcE [1])} \to \smash{\Gamma(\bigodot^2 \mbD^* [1])}$,\footnote{In fact, applying the theorem gives that $\Theta_0$ is an operator $\smash{\Gamma(\bigwedge^2 \mbD)} \to \smash{\Gamma(\bigodot^2 \mbD^* \otimes \bigwedge^2 \mbD)}$, but we establish in Section \ref{section:conformal-structure} the natural isomorphism $\smash{\bigwedge^2 \mbD} \cong \mcE[1]$.} so $\Theta_0(\sigma) = 0$ is an overdetermined system of just $\dim \smash{\bigodot^2 \mbD^*[1]} = 3$ p.d.e.s in $\sigma$, and thus it appears potentially weaker.

In fact the latter system is not weaker: We show that the normal Cartan connection $\omega$ satisfies a normalization condition that guarantees the induced connection $\nabla^{\mcV}$ is the preferred prolongation connection for $\Theta_0$ in the sense of \cite[Theorems 1.2, 1.3]{HSSS} and hence prove the following result.
\begin{theorem}\label{theorem:normality-of-solutions}
For any oriented $(2, 3, 5)$-distribution $\mbD$, all solutions of the BGG equation $\Theta_0(\sigma)=0$ are normal; equivalently, $\ker \Theta_0 = \ker \wt\Theta_0$.
\end{theorem}

We give our proof immediately; it relies only on algebraic aspects of the general theory of parabolic geometry (see the standard reference for \cite[\S3]{CapSlovak}, which defines the standard notation used in the proof), and in particular it does not use any explicit formula for $\Theta_0$.

\begin{proof}[Proof of Theorem \ref{theorem:normality-of-solutions}]
It suffices to show that for any regular, normal Cartan geometry $(\mcG \to M, \omega)$ of type $(\G_2, P)$ the curvature $\Omega\in\Omega^2(M,\mathcal{A})$ of the normal tractor connection $\nabla^{\mathcal{V}}$ on the standard tractor bundle $\mathcal{V}=\mathcal{G}\times_{P}\mathbb{V}$ satisfies $(\partial^*_{\mathcal{V}}\otimes\id_{\mathcal{V}^*})(\Omega)=0$, equivalently $\partial^*_{\mathcal{V}}(\Omega(t))=0$ for all $t\in\Gamma(\mathcal{\mcV})$: It then follows from \cite[Theorem 1.2]{HSSS} that $\nabla^{\mcV}$ coincides with the canonical prolongation connection for the corresponding BGG equation, and by \cite[Theorem 1.3]{HSSS} this means that all solutions are normal. (Here $\partial^*_{\mcV} : T^*M \otimes \mcV \to \mcV$ is the bundle map induced by the Kostant codifferential associated to $\bbV$, and $\mathcal A$ denotes the adjoint tractor bundle $\mcG \times_P \mfg_2$.)

The bundle map $\partial^*_{\mathcal{V}}\otimes\id_{\mathcal{V}^*}$ is induced from a $P$-equivariant map whose kernel describes a $P$-submodule $\mathbb{W}\subset\smash{\bigwedge^2\mathfrak{g}_{-}^*\otimes \End(\mathbb{V})}$. Let $\mathbb{E}:=\mathbb{W}\cap \ker \,\partial^*_{\mathfrak{g}}\subset \smash{\bigwedge^2\mathfrak{g}_{-}^*\otimes \mathfrak{g}}$ be the intersection of $\mathbb{W}$ with the kernel of $\partial^*_{\mathfrak{g}}$. Then, since regular, normal parabolic geometries associated with $(2,3,5)$ distributions are torsion-free,  \cite[Corollary 3.2]{CapCorr}  shows that if the harmonic part $A$ of the curvature function takes values in $\mathbb{E}\cap \ker \square$, then  the full curvature function $\kappa$ takes values in $\mathbb{E}$. 

Kostant's Theorem \cite[\S3.3]{CapSlovak} identifies the harmonic curvature module $\ker \square \subset (\mathfrak{g}_{-1})^*\otimes(\mathfrak{g}_{-3})^*\otimes\mathfrak{g}_0$. The Lie bracket defines an isomorphism $\mathfrak{g}_{-3}\cong\mathfrak{g}_{-1}\otimes\mathfrak{g}_{-2}$ and an inclusion $\mathfrak{g}_0\hookrightarrow (\mathfrak{g}_{-1})^*\otimes\mathfrak{g}_{-1}$.
If we apply  these identifications and view  $\ker \square$ as a subspace of $$(\mathfrak{g}_{-1}^*\otimes\mathfrak{g}_{-1}^*\otimes\mathfrak{g}_{-1}\otimes\mathfrak{g}_{-1}^*)\otimes\mathfrak{g}_2,$$ then Kostant's theorem shows that an element $A_{\alpha \beta}{}^{\gamma}{}_{\zeta}\in\ker \square$ is symmetric in $\alpha, \beta, \zeta$ and any trace (hence all traces) vanish:
$A_{\alpha \beta}{}^{\gamma}{}_{\gamma}=0$.

Now consider the action of  $A_{\alpha \beta}{}^{\gamma}{}_{\zeta}$ on $\mathbb{V}=\mathbb{V}_{-2}\oplus \mathbb{V}_{-1}\oplus\mathbb{V}_0\oplus\mathbb{V}_{1}\oplus\mathbb{V}_{2}$. On any of the $1$-dimensional  $G_0$-representations, $\mathbb{V}_2, \mathbb{V}_0,$ and $\mathbb{V}_{-2}$, the action is, up to multiplication with constants, given by $\lambda \mapsto A_{\alpha \beta}{}^{\gamma}{}_{\gamma} \lambda$ and thus trivial. 
On the $2$-dimensional $G_0$-representation $\mathbb{V}_1$
the action is given by $\phi^{\zeta} \mapsto A_{\alpha \beta}{}^{\gamma}{}_{\zeta}\phi^{\zeta}$ and on $\mathbb{V}_{-1}$ by $\tau_{\gamma} \mapsto -A_{\alpha \beta}{}^{\gamma}{}_{\zeta} \tau_{\gamma}$. Again using the symmetry properties 
of $A_{\alpha \beta}{}^{\gamma}{}_{\zeta},$ one sees that the Kostant codifferential $\partial_{\mathbb{V}}^*$ applied to these elements vanishes.
\end{proof}

This theorem has several consequences. Foremost, by the observations before the theorem this establishes that $\ker \Theta_0 = \ker \wt\Theta_0$, which in turn furnishes a concrete geometric interpretation of $\Theta_0$.
\begin{theorem}\label{theorem:geometric-characterization}
Let $(M, \mbD)$ be an oriented $(2, 3, 5)$ distribution. A section $\sigma \in \Gamma(\mcE[1])$ is a solution of the BGG equation $\Theta_0(\sigma) = 0$ iff it is an almost Einstein scale of the induced conformal structure $\mbc_{\mbD}$.
\end{theorem}
We thus call $\Theta_0$ the \textit{$(2, 3, 5)$ almost Einstein operator}, we call any solution of $\Theta_0(\sigma) = 0$ an \textit{almost Einstein scale} for the distribution $\mbD$, and we say that $\mbD$ is \textit{almost Einstein} iff that equation admits a nonzero solution: Theorem \ref{theorem:geometric-characterization} asserts that $\sigma$ is an almost Einstein scale for $\mbD$ iff it is an almost Einstein scale for $\mbc_{\mbD}$ and hence that $\mbD$ is almost Einstein iff $\mbc_{\mbD}$ is.

In turn, the characterization in Theorem \ref{theorem:geometric-characterization} shows that the problem of existence of Einstein representatives in the induced conformal class is a basic feature of the geometry of oriented $(2, 3, 5)$ distributions, and that it can be apprehended directly, that is, without explicit reference to the induced conformal structure.

Every solution of $\Theta_0(\sigma) = 0$ is the projection $\Pi_0(t)$ of a $\nabla^{\mcV}$-parallel section $t \in \Gamma(\mcV)$, and so a nonzero solution determines a reduction of the holonomy of $\nabla^{\mcV}$. Translating from \cite[Proposition A]{SagerschnigWillse} and the following text immediately gives the following.

\begin{proposition}
If $\sigma$ is a nonzero solution of the BGG equation $\Theta_0(\sigma) = 0$, then the holonomy $\Hol(\nabla^{\mcV})$ admits a reduction to $\SU(1, 2)$, $\SL(2, \bbR) \ltimes P_+$, or $\SL(3, \bbR)$ when the corresponding parallel tractor is, respectively, spacelike, isotropic, or timelike. Here, $SL(2, \bbR) \ltimes P_+ < P$ is the subgroup of $G_2$ preserving an isotropic vector in the standard representation $\bbV$.
\end{proposition}

When studying this problem of the existence of almost Einstein scales in a $(2, 3, 5)$ conformal structure, working with the operator $\Theta_0$ instead of $\wt\Theta_0$ has two practical advantages: Firstly, as we observed above, the system $\Theta_0(\sigma) = 0$ consists of many fewer equations than $\wt\Theta_0(\sigma) = 0$ does. The operator $\Theta_0$ is also lower-order in $\mbD$: It depends on just $4$ derivatives of a distribution $\mbD$ (see \eqref{equation:DQQ}-\eqref{equation:DXX}), whereas the construction $\mbD \rightsquigarrow \mbc_{\mbD} \rightsquigarrow \wt\Theta_0$ depends on $6$ derivatives of $\mbD$. Secondly, it does not require computing the conformal structure, which can be technically involved. This makes easier both determining the space $\ker \Theta_0 = \ker \wt\Theta_0$ of almost Einstein scales for any particular $(2, 3, 5)$ distribution (as we do in Examples \ref{example:flat-model}, \ref{example:Fq}, \ref{example:rolling-distribution}) and constructing new examples of almost Einstein $(2, 3, 5)$ distributions (as in Example \ref{example:31}).

On the other hand, if the induced conformal structure is available, we can use it to compute $\Theta_0$ immediately, which may be easier than using the formulae given in Section \ref{section:BGG-operator}:
\begin{proposition}\label{proposition:pullback}
Let $(M, \mbD)$ be an oriented $(2, 3, 5)$ distribution. The $(2, 3, 5)$ almost Einstein operator $\Theta_0$ satisfies
\[
	\Theta_0 = \iota^* \wt\Theta_0 ,
\]
where $\iota$ is the inclusion $\mbD \hookrightarrow TM$ and $\wt\Theta_0$ is the (conformal) almost Einstein operator for the induced conformal structure $\mbc_{\mbD}$.
\end{proposition}

Section \ref{section:geometry} reviews some basic features about the geometry of $(2, 3, 5)$ distributions, introduces a variation of index notation adapted to such distributions, and develops the tractor calculus for the standard tractor bundle associated to that geometry. Section \ref{section:BGG-operator} employs that calculus to derive an explicit formula for $\Theta_0$ in terms of a partial connection $\nabla$ and the lowest-homogeneity component of $\sfP$. Section \ref{section:conformal-structure} relates the objects we have constructed for a $(2, 3, 5)$ distribution to objects associated to the induced conformal structure. Section \ref{section:Monge-normal-form} recalls the (local) Monge normal form that expresses a $(2, 3, 5)$ distribution in terms of a real-valued function $F$ of five variables and give formulae for $\nabla$, the lowest-homogeneity component of $\sfP$, and $\Theta_0$ in a preferred scale and frame determined by the normal form. Finally, in Section \ref{section:examples}, we use $\Theta_0$ to recover efficiently some known examples of almost Einstein scales for particular $(2, 3, 5)$ distributions and to produce a new example of an almost Einstein $(2, 3, 5)$ distribution with properties not exhibited before. We also give a nonexample: We show for a particular homogeneous $(2, 3, 5)$ distribution that $\Theta_0$ has no solutions and exploit this to show quickly that the holonomy of the conformal structure it induces has holonomy equal to $\G_2$.

Ian Anderson's Maple package $\mathtt{DifferentialGeometry}$ was used to compute the formulae in Section \ref{section:Monge-normal-form} and to carry out computations for examples in Section \ref{section:examples}. 

\thanks{It is a pleasure to thank Mike Eastwood for several helpful conversations.

The first author is an INdAM (Istituto Nazionale di Alta Matematica) research fellow. She gratefully acknowledges support from the Austrian Science Fund (FWF) via project J3071--N13 and support from project FIR--2013 Geometria delle equazioni differenziali. The second author gratefully acknowledges support from the Australian Research Council and the Austrian Science Fund (FWF), the latter via project P27072--N25.}

\section{The geometry of $(2, 3, 5)$ distributions}\label{section:geometry}

\subsection{The Levi bracket}

Recall that a $(2, 3, 5)$ distribution $\mbD$ on a manifold $M$ determines a filtration
\[
	\mbD \subset [\mbD, \mbD] \subset TM
\]
of the tangent bundle.

We will need certain maps canonically determined by such a distribution: The Lie bracket of vector fields induces a natural vector bundle map $\mbD \times \mbD \to [\mbD, \mbD] / \mbD$ defined by $(\xi_x, \eta_x) \mapsto q_{-2}([\xi, \eta]_x)$, where $q_{-2} : [\mbD, \mbD] \to [\mbD, \mbD] / \mbD$ is the canonical quotient map, and this map descends to a natural isomorphism
\[
	\mcL: {\textstyle \bigwedge^2 \mbD} \stackrel{\cong}{\to} [\mbD, \mbD] / \mbD =: \gr_{-2}(TM) .
\]
It is also convenient to denote
\[
	\mcE[1] := \gr_{-2} (TM) \cong {\textstyle \bigwedge^2 \mbD} .
\]
Similarly, the Lie bracket induces a vector bundle isomorphism
\[
	\mcL: \mbD \otimes ([\mbD, \mbD] / \mbD) \stackrel{\cong}{\to} TM / [\mbD, \mbD] =: \gr_{-3}(TM),
		\qquad
	\mathcal{L}(\xi_x,q_{-2}(\zeta_x)) := q_{-3}([\xi,\zeta]_x),
\]
where $q_{-3}: TM\to TM / [\mbD, \mbD]$ denotes the canonical quotient map. The bundle isomorphisms are components of the \textit{Levi bracket}.


\subsection{Adapted index notation}\label{subsection:adapted-index-notation}
To express sections of the graded bundles $\gr_{-i}(TM)$ and their duals and tensor products with weight bundles, we use use the following notation: We write a section $X \in \Gamma(\gr_{-1}(TM)) \cong \Gamma(\mbD)$ with an upper Greek index, as $X{}^{\alpha}$, a section $r \in \Gamma(\gr_{-2}(TM))$ with an upper solid lozenge, $\smash{r^{\blacklozenge}}$, or, when it causes no ambiguity, without an index, $r$, and a section $Y \in \Gamma(\gr_{-3}(TM)) \cong \Gamma(TM / [\mbD, \mbD])$ with a lower barred Greek index, as $Y_{\bar\alpha}$. Likewise, we write a section $Z \in \Gamma(\mbD^*)$ with a lower Greek index, $Z_{\alpha}$, $s \in \Gamma(([\mbD, \mbD] / \mbD)^*)$ with a lower solid lozenge, $\smash{s_{\blacklozenge}}$, or without an index, and $W \in \Gamma((TM / [\mbD, \mbD])^*)$ with an upper barred Greek index, $W^{\bar\alpha}$.

As usual, we indicate the canonical pairing between a section of one of these bundles and a section of its dual by repeating the indices, one upper and one lower, for example, $(X^{\alpha}, Z_{\alpha}) \mapsto Z_{\alpha} X^{\alpha}$. Just as for the usual use of tangent indices, we can write multiple indices to indicate sections of appropriate (possibly weighted) tensor bundles, for example $\omega^{\alpha \beta} \in \Gamma(\mbD \otimes \mbD [w])$. Then, we may regard the Levi bracket component $\smash{\mcL: {\textstyle \bigwedge^2 \mbD} \to [\mbD, \mbD] / \mbD}$ as a section $\mcL_{\alpha \beta} \in \smash{\Gamma(\bigwedge^2 \mbD^*[1])}$ or just as well a section $\mcL^{\alpha \beta} \in \smash{\Gamma(\bigwedge^2 \mbD[-1])}$, which we normalise to satisfy $\mcL^{\gamma \alpha} \mcL_{\gamma \beta} = \delta^{\alpha}{}_{\beta}$. Similarly, we may identify the component $\smash{\mcL: \mbD \otimes ([\mbD, \mbD] / \mbD) \to TM / [\mbD, \mbD]}$ as a section $\smash{\mcL_{\bar\alpha \beta} \in \Gamma(\mbD^* \otimes (TM / [\mbD, \mbD])[-1])}$ or a section $\smash{\mcL^{\bar\beta \alpha} \in \Gamma(\mbD \otimes (TM / [\mbD, \mbD])^*[1])}$, which we normalise to satisfy $\mcL^{\bar\gamma \alpha} \mcL_{\bar\gamma \beta} = \delta^{\alpha}{}_{\beta}$.\footnote{We may view $\mcL^{\alpha \beta}$ as a bundle isomorphism $\smash{\bigwedge^2 \mbD^* \to ([\mbD, \mbD] / \mbD)^*}$; this is $-\frac{1}{8}$ times the component induced by the algebraic bracket component $\smash{\bigwedge^2 \mfg_{+1} \to \mfg_{+2}}$. Likewise, $\smash{\mcL^{\bar\alpha \beta}}$, viewed as a bundle isomorphism $\smash{\mbD^* \otimes ([\mbD, \mbD] / \mbD)^* \to (TM / [\mbD, \mbD])^*}$, is $-\frac{2}{9}$ times the algebraic bracket component induced by the map $\mfg_{+1} \times \mfg_{+2} \to \mfg_{+3}$.}

We can use the canonical sections $\mathcal{L}_{\alpha\beta}, \mcL^{\alpha\beta}, \mathcal{L}_{\bar\alpha \beta}, \mcL^{\bar\alpha \beta}$ to change arbitrarily the position and/or barredness of an index at the cost of an appropriate change of weight and coefficient. For example, we may identify $\xi^{\alpha} \in \Gamma(\mbD)$ with a section $\xi_{\bar\alpha} := \mcL_{\bar\alpha \beta} \xi^{\beta} \in \Gamma((TM / [\mbD, \mbD])[-1])$, a section $\xi_{\alpha} := \mcL_{\alpha\beta} \xi^{\beta} \in \Gamma(\mbD^*[1])$, or a section $\xi^{\bar\alpha} := - \frac{3}{4} \mcL^{\bar\alpha \gamma} \mcL_{\beta\gamma} \xi^{\beta} \in \Gamma((TM / [\mbD, \mbD])^*[2])$.\footnote{The occurrence of the factor $\smash{-\frac{3}{4}}$ here is a consequence of the choice of normalization of $\smash{\mcL^{\bar\alpha \beta}}$ in the previous paragraph.}


\subsection{Weyl structures}
\label{subsection:Weyl-structures}
Let $\mbD$ be an oriented $(2,3,5)$ distribution,  $(\mathcal{G}\to M, \omega)$ the canonically associated regular, normal parabolic geometry of type $(\G_2, P)$, and $\mathcal{G}_0=\mathcal{G}/P_+$  the underlying graded frame bundle of $\mbD$. A (local) \emph{Weyl structure} of the geometry $(\mathcal{G}\to M, \omega)$ is a (local) $G_0$-equivariant section $$s:\mathcal{G}_0\to\mathcal{G}$$ of the projection $\mathcal{G}\to\mathcal{G}_0$ \cite{CapSloWeyl, CapSlovak}. The pullback of the Cartan connection by a Weyl structure $$s^*\omega=s^*\omega_{-}+s^*\omega_0+s^*\omega_+\in\Omega^1(\mathcal{G}_0,\mathfrak{g}_{-}\oplus\mathfrak{g}_0\oplus\mathfrak{g}_+)$$ decomposes into  three components:
\begin{itemize}
\item The \emph{soldering form} $s^*\omega_{-}\in\Omega^1(\mathcal{G}_0,\mathfrak{g}_-)$ is a $G_0$-equivariant, horizontal $1$-form that can be equivalently viewed as an isomorphism
$$(\pi_{-3},\pi_{-2},\pi_{-1}): TM\to\gr(TM)=\gr_{-3}(TM)\oplus\gr_{-2}(TM)\oplus\gr_{-1}(TM)$$
splitting the filtration $\mbD \subset [\mbD, \mbD] \subset TM$. Here,
$\gr_{-1}(TM)=\mbD\cong\mathcal{G}_0\times_{G_0}\mathfrak{g}_{-1},$ $\gr_{-2}(TM)=[\mbD,\mbD]/\mbD\cong\mathcal{G}_0\times_{G_0}\mathfrak{g}_{-2},$ and $\gr_{-3}(TM)=TM/[\mbD,\mbD]\cong\mathcal{G}_0\times_{G_0}\mathfrak{g}_{-3}.$
\item The \emph{Weyl connection} $s^*\omega_0\in\Omega^1(\mathcal{G}_0,\mathfrak{g}_0)$ is a principal connection on $\mathcal{G}_0\to M$. It induces a linear connection on each  bundle associated to $\mathcal{G}_0$ and, via the soldering form, a linear connection on $TM\cong\gr(TM)\cong\mathcal{G}_0\times_{G_0}\mathfrak{g}_{-}$. All of these induced connections shall be denoted by $\nabla$ and referred to as Weyl connections.
\item The \emph{Rho tensor} $s^*\omega_{+}\in\Omega^1(\mathcal{G}_0,\mathfrak{g}_+)$  is a $G_0$-equivariant, horizontal $1$-form and can thus be  viewed as a $1$-form $\sfP\in\Omega^1(M,\gr(T^*M))$, where $\gr(T^*M)=\bigoplus_{i = 1}^{3}\mathrm{gr}_i(T^*M)=\mathcal{G}_0\times_{G_0}\mathfrak{g}_{+}$ and $\mathrm{gr}_i(T^*M)$ is naturally isomorphic to $\mathrm{gr}_{-i}(TM)^*$.
\end{itemize}

There is a bijective correspondence between the set of all Weyl structures and $\Gamma(\gr(T^*M))$. Having fixed one Weyl structure $s:\mathcal{G}_0\to\mathcal{G}$, the correspondence is given by mapping $\Upsilon=(\Upsilon_1,\Upsilon_2,\Upsilon_3)\in\Gamma(\gr(T^*M))$  to the Weyl structure $\hat{s}(u)=s(u) \exp \Upsilon_1(u) \exp \Upsilon_2(u) \exp \Upsilon_3(u),$ where we view $\Upsilon_i\in\Gamma(\gr_i(T^*M))$ as a $G_0$-equivariant function $\Upsilon_i:\mathcal{G}_0\to\mathfrak{g}_i$.

In particular, there is a bijective correspondence between Weyl structures of $(\mathcal{G}\to M, \omega)$ and linear connections on the bundle $\mathcal{E}[1],$ 
which in one direction  is given by mapping  $s:\mathcal{G}_0\to\mathcal{G}$ to the induced Weyl connection $\nabla$ on $\mathcal{E}[1]$ \cite[Corollary 5.1.6]{CapSlovak}. A scale, i.e., a nowhere vanishing section $\theta \in \Gamma(\mathcal{E}[1]),$ 
 determines a flat connection of $\mathcal{E}[1]$  and thus by the aforementioned bijective correspondence a Weyl structure of the geometry.\footnote{The term \textit{scale} here and the term \textit{almost Einstein scale} (for a conformal structure) mentioned in the introduction are both standard in the literature, but they suffer a mild inconsistency: An \textit{almost Einstein scale} of a conformal structure (respectively, oriented $(2, 3, 5)$ distribution) is a section of $\mcE[1]$ in the kernel of the almost Einstein operator $\wt\Theta_0$ (respectively, $\Theta_0$), but by definition such a section is a \textit{scale} iff it vanishes nowhere. Thus, not all almost Einstein scales are scales.} Weyl structures arising that way are called exact Weyl structures.
In \cite{thesis} (see also \cite{CapSagerschnig}), the  Weyl connection, the soldering form, and the Rho tensor  determined by a scale were constructed.

Fix a scale $\theta \in \Gamma(\mcE[1])$. Then there is a unique vector field $R\in\Gamma([\mbD,\mbD])$  and  a unique partial connection $\nabla:\Gamma(\mbD)\times\Gamma(\mbD)\to\Gamma(\mbD)$ such that $q_{-2}(R)=\theta$,
\[
\mathcal{L}(\nabla_{\gamma}\xi,\theta)=q_{-3}([\gamma,[\xi,R]]) \textrm{,}
\]
and 
\[
\nabla_{\gamma}\theta
	=\mathcal{L}(\nabla_{\gamma}\xi,\eta)+\mathcal{L}(\xi,\nabla_{\gamma}\eta)=0 
\]
for all $\xi, \eta, \gamma\in\Gamma(\mbD)$ with  $\mathcal{L}(\xi,\eta)= \theta$.  The field $R$ is called the \textit{generalised Reeb field} associated to $\theta$, and the partial connection $\nabla$ is part of the Weyl connection determined by $\theta$. 
We shall denote by $\alpha\in\Gamma(T^* M)$  the unique  $1$-form such that 
\[
\alpha(R)=1\textrm{,} \quad \textrm{and} \quad \alpha(\xi)=0 \quad \textrm{and} \quad \alpha([\xi,R])=0 \quad \textrm{for all} \quad \xi\in\Gamma(\mbD) .
\]
Moreover, we introduce two maps, $\Psi_1, \Psi_2:\Gamma(\mbD)\to\Gamma(\mbD)$, defined by
\[
	\mathcal{L}(\Psi_1(\gamma),\theta) = \tfrac{1}{2} q_{-3}([R,[\gamma,R]]),\qquad
\Psi_2(\gamma)=\alpha([\eta,[\gamma,R]])\xi-\alpha([\xi,[\gamma,R]])\eta ,
\]
where $\xi, \eta, \gamma\in \Gamma(\mbD)$ and $\mathcal{L}(\xi,\eta)=\theta$. Both maps satisfy the Leibniz Rule $\Psi_i(f\gamma)=f\,\Psi_i(\gamma)+(R\cdot f)\gamma$ for all $f \in C^{\infty}(M)$ and $\gamma \in \Gamma(\mbD)$.\footnote{In \cite{CapSagerschnig} $\Psi_1$ is called $\Psi$ and instead of $\Psi_2$ the map  $\Phi(\gamma):=-\nabla_{\xi}\nabla_{\eta}\gamma+\nabla_{\eta}\nabla_{\xi}\gamma+\nabla_{\pi_{-1}([\xi,\eta])}\gamma=-\Psi_2(\gamma)+2\Psi_1(\gamma)$ is used.}
 Then the  splitting $(\pi_{-3},\pi_{-2},\pi_{-1}):TM\to\gr(TM)$ determined by $\theta$ is given as follows:   $\pi_{-3}=q_{-3}:TM\to\gr_{-3}(TM)$ is the natural quotient map,  $\pi_{-2}:TM\to \gr_{-2}(TM)$ is given by  $\pi_{-2}(\zeta)=\alpha(\zeta)\theta$, and $\pi_{-1}:TM\to\mbD$ is   characterised by the conditions  
\begin{align*} \pi_{-1}(\xi)=\xi,\quad   \pi_{-1}(R)=0,\quad \pi_{-1}([R,\xi])=\tfrac{3}{5}\Psi_1(\xi)+\tfrac{2}{5}\Psi_2(\xi),
\end{align*}
  for all $\xi\in\Gamma(\mbD)$. 
The lowest-homogeneity component of the Rho tensor is a section of $\smash{\bigodot^2\mbD^*}$ given by
\begin{align}\label{RhoT}
\sfP(\xi)(\eta)=\tfrac{1}{5}\mathrm{d}\alpha(\Psi_1(\xi)-\Psi_2(\xi),\eta).
\end{align} 
Expressions for the remaining parts of the (exact) Weyl connection and the Rho tensor can be found in \cite{thesis}.


\subsection{Tractor calculus}
\label{section:tractor-calculus}
The $P$-invariant filtration of the standard representation $\mathbb{V}$ (see the appendix), determines a filtration 
\begin{equation}\label{filt}
\mcV = \mcV^{-2} \supset \mcV^{-1} \supset \mcV^0 \supset \mcV^1 \supset \mcV^2
\end{equation}
of the \emph{tractor bundle} $$\mathcal{V}:=\mathcal{G}\times_P\mathbb{V}.$$
There are natural identifications
$\mathcal{V}^{-2}/\mathcal{V}^{-1}\cong \mathcal{E}[1]$, $\mathcal{V}^{-1}/\mathcal{V}^0\cong \gr_{-3}(TM)[-1],$
$\mathcal{V}^0/\mathcal{V}^1\cong \gr_{-2}(TM)[-1],$
$\mathcal{V}^1/\mathcal{V}^2\cong \gr_{-1}(TM)[-1],$
and
$\mathcal{V}^2\cong\mathcal{E}[-1]$. Here, $\mcE[w]$ denotes $\mcE[1]^{\otimes w}$---in particular, $\mcE[-1] = \mcE[1]^* \cong \smash{\bigwedge^2 \mbD^*}$---and for a vector bundle $B \to M$ we denote $B[w] := B \otimes \mcE[w]$.
 
We encode these identifications in the composition series
\begin{equation}\label{comp}\mathcal{V}\cong \mathcal{E}[1] \fplusl \gr_{-3}(TM)[-1]\fplusl\gr_{-2}(TM)[-1]\fplusl\gr_{-1}(TM)[-1]\fplusl\mathcal{E}[-1].\end{equation}
Since $\mathcal{V}$ is the associated bundle corresponding to  a $\G_2$-representation (restricted to $P$), the regular, normal Cartan connection $\omega\in\Omega^1(\mathcal{G},\mathfrak{g}_2)$ induces a linear connection $\nabla^{\mathcal{V}}$ on $\mathcal{V}$, called the \emph{(normal) tractor connection}.

A choice of a scale $\theta$ determines an identification of $\mcV$ with its associated graded bundle,
\begin{equation}\label{comp_splitted}\mathcal{V}\cong  \mathcal{E}[1]\oplus\gr_{-3}(TM)[-1]\oplus\gr_{-2}(TM)[-1]\oplus\gr_{-1}(TM)[-1]\oplus\mathcal
{E}[-1],\end{equation}
i.e., a splitting of \eqref{filt} \cite[Corollary 5.1.3]{CapSlovak}. So, a section $t\in\Gamma(\mathcal{V})$ can be written as a $5$-tuple
\[t \stackrel{\theta}{=}
    \left(
        \begin{array}{c}
            \chi \\
            \phi^{\alpha} \\
            \upsilon \\
            \tau_{\alpha} \\
            \sigma \\
        \end{array}
    \right)
    		\in \Gamma
    \left(
        \begin{array}{c}
            \mcE[-1] \\
            \gr_{-1}(TM)[-1] \\
            \gr_{-2}(TM)[-1] \\
            \gr_{-3}(TM)[-1] \\
            \mcE[1] \\
        \end{array}
    \right)
    \textrm{.}
\]
Exploiting the Levi bracket component $\mcL_{\bar\alpha \beta}$ identifies the component $\gr_{-3}(TM)[-1]$ with $\gr_{-1}(TM) \otimes \gr_{-2}(TM)[-1] \cong \smash{{\textstyle (\mbD \otimes \bigwedge^2 \mbD)} \otimes {\textstyle \bigwedge^2 \mbD^*}} \cong \mbD$. On the other hand, the explicit form \eqref{table} of the representation $\bbV$ suggests that we view this bundle as $\smash{\mbD^* \otimes \bigwedge^2 \mbD} \cong \mbD^* [1]$ and so mark a section $\tau \in \gr_{-3}(TM)[-1]$ with a lower (unbarred) index. Similarly, we identify $\gr_{-1}(TM)[-1]$ with $\mbD [-1]$ and so mark a section $\phi$ thereof with an upper (again unbarred) index.

We further introduce the notation  $\sfP_{\alpha \beta}, \sfP_{\blacklozenge\alpha}, \sfP_{\alpha\blacklozenge}, \sfP_{\alpha}{}^{\bar\beta}, \sfP^{\bar\alpha}{}_{\beta}, \sfP_{\blacklozenge\blacklozenge}, \sfP_{\blacklozenge}{}^{\bar\alpha}, \sfP^{\bar\alpha}{}_{\blacklozenge}, \sfP^{\bar \alpha \bar \beta}$ for the components of the Rho tensor $\sfP$; these quantities are characterised by the decomposition formula
\begin{alignat}{4}
\sfP(\xi)(\eta)
	=  {}\phantom{+}{}&\sfP_{\alpha\beta}(\xi_{-1})^{\alpha}(\eta_{-1})^{\beta}
	  &{}+{}&\tfrac{1}{2}\sfP_{\alpha\blacklozenge}(\xi_{-1})^{\alpha}(\eta_{-2})^{\blacklozenge}
      &{}+{}&\tfrac{1}{3}\sfP_{\alpha}{}^{\bar\beta}(\xi_{-1})^{\alpha}(\eta_{-3})_{\bar\beta} \nonumber \\
	   {}+{}&\sfP_{\blacklozenge\beta}(\xi_{-2})^{\blacklozenge}(\eta_{-1})^{\beta}
	  &{}+{}&\tfrac{1}{2}\sfP_{\blacklozenge\blacklozenge}(\xi_{-2})^{\blacklozenge}(\eta_{-2})^{\blacklozenge}
	  &{}+{}&\tfrac{1}{3}\sfP_{\blacklozenge}{}^{\bar\beta}(\xi_{-2})^{\blacklozenge}(\eta_{-3})_{\bar\beta} \label{equation:Rho-components} \\
       {}+{}&\sfP^{\bar\alpha}{}_{\beta}(\xi_{-3})_{\bar\alpha}(\eta_{-1})^{\beta}
      &{}+{}&\tfrac{1}{2}\sfP^{\bar\alpha}{}_{\blacklozenge}(\xi_{-3})_{\bar\alpha}(\eta_{-2})^{\blacklozenge}
      &{}+{}&\tfrac{1}{3}\sfP^{\bar\alpha \bar\beta}(\xi_{-3})_{\bar\alpha}(\eta_{-3})_{\bar\beta} \nonumber,
\end{alignat}
where $\xi \stackrel{\theta}{=}(\xi_{-3},\xi_{-2},\xi_{-1})$ and $\eta \stackrel{\theta}{=}(\eta_{-3},\eta_{-2},\eta_{-1})$.\footnote{
Here, we view $\eta$ as a section of $\Gamma(\gr (TM))$ and hence $\sfP(\xi) \in \Gamma(\gr(T^*M))$ as a section of $\Gamma((\gr (TM))^*)$ via the canonical identification induced by the isomorphism $\mfp_+ \cong (\mfg_2 / \mfp)^*$ determined by the Killing form on $\mfg_2$. The pairings $\mfg_{-i} \times \mfg_{+i} \to \Bbb R$ it induces are recorded in the appendix. We use the coefficients in those pairings for the coefficients of the symbols $\sfP_{\alpha \beta}$, etc., so that the symbols satisfy the symmetry identities $\smash{\sfP_{\alpha\blacklozenge} = \sfP_{\blacklozenge\alpha}}$, $\smash{P_{\alpha}{}^{\bar\beta} = \sfP^{\bar\beta}{}_{\alpha}}$, $\smash{\sfP^{\bar\alpha}{}_{\blacklozenge} = \sfP_{\blacklozenge}{}^{\bar\alpha}}$. See Proposition \ref{proposition:Levi-Civita} and the following discussion.}

\begin{proposition}\label{proposition:splitting}
Fix an oriented $(2, 3, 5)$ distribution.
\begin{enumerate}
\item Under a change $\theta \rightsquigarrow \hat\theta$ of scale, say, with corresponding change $$s \rightsquigarrow \hat{s}=s \exp \Upsilon_1 \exp \Upsilon_2 \exp \Upsilon_3$$ of Weyl structure, the splitting transforms as follows:
\begin{equation*}\label{changeofsplitting}
   t \stackrel{\hat{\theta}}=\left(
        \begin{array}{c}
           \hat \chi \\
            \hat\phi^{\alpha} \\
            \hat\upsilon\\
            \hat\tau_{\alpha} \\
            \hat\sigma \\
        \end{array}
    \right)
    =
        \left(\begin{array}{c}
		\left(
			\begin{array}{c}
				  \chi
				- \phi^{\alpha}(\Upsilon_1)_{\alpha}
				- \upsilon\Upsilon_2-\tau_{\alpha}(\Upsilon_3)^{\alpha}
				+ 4\Upsilon_2\mathcal{L}^{\alpha\beta} (\Upsilon_1)_{\alpha}\tau_{\beta} \\
				{}+ \tfrac{1}{2}\sigma\Upsilon_2\Upsilon_2
				- \sigma(\Upsilon_1)_{\alpha}(\Upsilon_3)^{\alpha}
			\end{array}
			\right) \\
		\left(
			\begin{array}{c}
				  \phi^{\alpha}
				- 4\upsilon\mathcal{L}^{\beta\alpha}(\Upsilon_1)_{\beta}
				- 2\Upsilon_2\mathcal{L}^{\beta\alpha}\tau_{\beta}
				+8 \mathcal{L}^{\gamma\zeta}(\Upsilon_1)_{\gamma}\tau_{\zeta}\mathcal{L}^{\beta\alpha}(\Upsilon_1)_{\beta} \\
				{}+ \sigma(\Upsilon_3)^\alpha
				- 2\sigma\Upsilon_2\mathcal{L}^{\beta\alpha}(\Upsilon_1)_{\beta}
			\end{array}
		\right) \\

         \upsilon-4\mathcal{L}^{\alpha\beta}(\Upsilon_1)_{\alpha}\tau_{\beta}-\sigma\Upsilon_2\\
         \tau_{\alpha}+\sigma(\Upsilon_1)_{\alpha} \\
           \sigma \\
     \end{array}
    \right).
 \end{equation*}

\item In terms of the splitting, the Weyl connection $\nabla$, and the Rho tensor $\sfP$ determined by $\theta$, the tractor connection on $\mathcal{V}$ is given by  
\[\nabla^{\mathcal{V}}_{\beta}
    \left(
        \begin{array}{c}
            \chi \\
            \phi^{\alpha} \\
            \upsilon \\
            \tau_{\alpha} \\
            \sigma \\
                    \end{array}
    \right) =
     \left(
        \begin{array}{c}
            \nabla_\beta\chi +  \sfP_{\beta\zeta} \phi^{\zeta} + \sfP_{\beta\blacklozenge} \upsilon + \sfP_{\beta}{}^{\zeta} \tau_{\zeta}\\
            \nabla_{\beta}\phi^{\alpha} + \chi \delta^{\alpha}{}_{\beta}  +4\upsilon \mathcal{L}^{\gamma\alpha} \sfP_{\beta\gamma}+2 \sfP_{\beta\blacklozenge} \mathcal{L}^{\gamma\alpha} \tau_{\gamma}-\sigma \sfP_{\beta}{}^{\alpha}\\
            \nabla_{\beta}\upsilon + \tfrac{1}{2} \mathcal{L}_{\beta\gamma} \phi^{\gamma}+4 \mathcal{L}^{\gamma\zeta} \sfP_{\beta\gamma} \tau_{\zeta} +\sigma \sfP_{\beta\blacklozenge} \\
            \nabla_{\beta}\tau_{\alpha}+\tfrac{1}{2} \upsilon \mathcal{L}_{\beta\alpha} - \sigma \sfP_{\beta\alpha}\\
            \nabla_{\beta}\sigma-\tau_{\beta} \\
        \end{array}
    \right)
\]

\[\nabla^{\mathcal{V}}_{\blacklozenge}
    \left(
    \begin{array}{c}
            \chi \\
            \phi^{\alpha} \\
            \upsilon \\
            \tau_{\alpha} \\
            \sigma \\
                    \end{array}
            \right) =
     \left(
                \begin{array}{c}
           \nabla_{\blacklozenge} \chi+  \sfP_{\blacklozenge\gamma} \phi^{\gamma} + \sfP_{\blacklozenge\blacklozenge} \upsilon +   {\sfP}_{\blacklozenge}{}^{\gamma}\tau_{\gamma} \\
          \nabla_{\blacklozenge}  \phi^{\alpha} +4 \upsilon \mathcal{L}^{\gamma\alpha} \sfP_{\blacklozenge\gamma}+2 \sfP_{\blacklozenge\blacklozenge} \mathcal{L}^{\gamma\alpha} \tau_{\gamma} - \sigma {\sfP}_{\blacklozenge}{}^{\alpha}\\
          \nabla_{\blacklozenge}  \upsilon +  \chi +4 \mathcal{L}^{\gamma\zeta} \sfP_{\blacklozenge\gamma} \tau_{\zeta} + \sigma \sfP_{\blacklozenge\blacklozenge}\\
         \nabla_{\blacklozenge}   \tau_{\alpha}-\tfrac{1}{4}  \mathcal{L}_{\gamma\alpha} \phi^{\gamma}- \sigma \sfP_{\blacklozenge\alpha}\\
            \nabla_{\blacklozenge}\sigma +  \upsilon \\
        \end{array}
        \right)
\]

\[(\nabla^{\mathcal{V}})^{\beta}
    \left(
        \begin{array}{c}
            \chi \\
            \phi^{\alpha} \\
            \upsilon \\
            \tau_{\alpha} \\
            \sigma \\
        \end{array}
    \right) =
     \left(
        \begin{array}{c}
          \nabla^{\beta}  \chi +   \sfP^{\beta}{}_{\gamma} \phi^{\gamma}+ \sfP^{\beta}{}_{\blacklozenge} \upsilon +  {\sfP}^{\beta\gamma}\tau_{\gamma}\\
       \nabla^{\beta}     \phi^{\alpha} +4 \upsilon \mathcal{L}^{\gamma\alpha} {\sfP^{\beta}}_{\gamma}+2 \sfP^{\beta}{}_{\blacklozenge} \mathcal{L}^{\gamma\alpha} \tau_{\gamma}- \sigma {\sfP}^{\beta\alpha}\\
            \nabla^{\beta} \upsilon+4 \mathcal{L}^{\gamma\zeta} {\sfP^\beta}_{\gamma} \tau_{\zeta} +\sigma \sfP^{\beta}{}_{\blacklozenge} \\
          \nabla^{\beta}  \tau_{\alpha} + \chi\delta^{\beta}{}_{\alpha}- \sigma {\sfP^{\beta}}_{\alpha}\\
         \nabla^{\beta}   \sigma -\phi^{\beta} \\
        \end{array}
    \right) .
\]
\end{enumerate}
\end{proposition}

Throughout the formulae in Proposition \ref{proposition:splitting} we have used the identifications given at the end of Subsection \ref{subsection:adapted-index-notation} to change the barred indices that occur to unbarred ones (but we have not changed any upper indices to lower ones or vice versa).

\begin{proof}[Proof of Proposition \ref{proposition:splitting}]\phantom{}~

\begin{enumerate}
	\item This follows from the general formula for the change of the identification $\mathcal{V}\cong\gr(\mathcal{V})$ induced by a change of Weyl structures (see \cite[Proposition 5.1.5]{CapSlovak}) and the form of the $\mathfrak{g}_+$-action on $\mathbb{V}$, see \eqref{table}.
	\item Writing a section $t\in\Gamma(\mathcal{V})$ as $t \stackrel{\theta}{=}(t_{-2},\dots,t_2)$, a vector field $\xi\in\Gamma(TM)$ as  $\xi \stackrel{\theta}{=}(\xi_{-3},\xi_{-2},\xi_{-1})$, and the bundle map induced by the action of $\mathfrak{g}$ on $\mathbb{V}$ by $\bullet:\mathcal{A}\times\mathcal{V}\to\mathcal{V}$ (here, $\mcA := \mcG \times_P \mfg_2$ is the adjoint tractor bundle), the tractor connection is given by \cite[Proposition 5.1.10]{CapSlovak}
$$(\nabla^{\mathcal{V}}_{\xi}t)_i=\nabla_\xi t_i+\sum_{j=1}^{3}\sfP_j(\xi)\bullet t_{i-j}+\sum_{j=1}^{3}\xi_{-j}\bullet t_{i+j} .$$
Using the action of $\mathfrak{g}$ on $\mathbb{V}$, which we record in \eqref{table}, yields the given formulae. \qedhere
\end{enumerate}
\end{proof}

\section{The first standard BGG operator for $(2, 3, 5)$ distributions}
\label{section:BGG-operator}
For each irreducible representation of $\G_2$ on any manifold $M$ equipped with an oriented $(2, 3, 5)$ distribution $\mbD$ there is a corresponding sequence of invariant differential operators, called \emph{BGG operators} \cite{CSS, CalderbankDiemer}. Applying the Bott-Borel-Weil Theorem (see \cite[Theorem 3.3.5]{CapSlovak}) gives that for the standard representation $\bbV$ this sequence has the form
\[
	                                      \Gamma(                       \mcE   [ 1] )
	\mathop{\longrightarrow}^{\Theta_0}_2 \Gamma({\textstyle \bigodot^2 \mbD^* [ 1]})
	\mathop{\longrightarrow}^{\Theta_1}_3 \Gamma({\textstyle \bigodot^3 \mbD^*     })
	\mathop{\longrightarrow}^{\Theta_2}_4 \Gamma({\textstyle \bigodot^3 \mbD^* [-2]})
	\mathop{\longrightarrow}^{\Theta_3}_3 \Gamma({\textstyle \bigodot^2 \mbD^* [-4]})
	\mathop{\longrightarrow}^{\Theta_4}_2 \Gamma(                       \mcE   [-6] ) .
\]
(The numbers under the arrows respectively indicate the orders of the operators above them.) In this section we compute explicitly the first operator, $\Theta_0$.

Denote by $\Pi_0$ the canonical quotient map $\mathcal{V}\to \mathcal{V}/\mathcal{V}^{-1}=\mathcal{E}[1]$, 
$$
\Pi_0 :
	\begin{pmatrix}
           \chi \\
            \phi^{\alpha} \\
            \upsilon \\
            \tau_{\alpha} \\
            \sigma \\
	\end{pmatrix}
	\mapsto \sigma . 
    $$
Also, recall that the differential splitting operator $L_0: \Gamma(\mathcal{E}[1])\to\mathcal{V}$ is characterised by the  conditions (1) $\Pi_0 \circ L_0=\id_{\mcE[1]}$ and (2) $\partial^*\circ\nabla^{\mathcal{V}}\circ L_0=0$, where $\partial^*:T^*M\otimes\mathcal{V}\to\mathcal{V}$ is the Kostant codifferential, which is characterised by linearity and the formula $\partial^*(\phi\otimes t) := -\phi\bullet t$ on decomposable elements.\footnote{In this formula we regard $\phi \in \Gamma(T^*M)$ as an adjoint tractor via the canonical inclusion $T^*M \hookrightarrow \mcA$.} The \textit{first BGG operator} $\Theta_0:\Gamma(\mathcal{E}[1])\to \smash{\Gamma(\bigodot^2\mbD^*[1])}$ is then defined as 
$$\Theta_0 := \Pi_1 \circ\nabla^{\mathcal{V}}\circ L_0,$$ where $\Pi_1$ denotes the projection $\ker\,\partial^* \to \ker\,\partial^*/\im\,\partial^*\cong \smash{\bigodot^2\mbD^*[1]}$; the latter appearance of $\partial^*$ here refers to the Kostant codifferential $\smash{\bigwedge^2 T^*M \otimes \mcV \to T^*M \otimes \mcV}$. The overdetermined system $\Theta_0(\sigma)=0$ of partial differental equations is called a \textit{first BGG equation}.

Computing gives that $\partial^* \nabla^{\mcV}$ takes the form
\begin{equation}\label{equation:codifferential-nabla}
\partial^*\nabla^{\mathcal{V}}
		\begin{pmatrix}
            \chi \\
            \phi^{\alpha} \\
            \upsilon \\
            \tau_{\alpha} \\
            \sigma \\
		\end{pmatrix}
    =
    \begin{pmatrix}
        		\star \\
        		\star \\
        		-4 \mcL^{\beta \gamma} (\nabla_{\beta} \tau_{\gamma}  + \frac{1}{2} \upsilon \mcL_{\beta \gamma} - \sfP_{\beta \gamma} \sigma) - (\nabla_{\blacklozenge} \sigma + \upsilon)  \\
			\nabla_{\alpha}\sigma-\tau_{\alpha} \\
            0 \\
	\end{pmatrix} .
	\end{equation}
Computing from \cite[(3.26)]{thesis} gives that $\nabla_{\blacklozenge} \sigma = \mcL^{\beta \gamma} (\nabla_{\beta} \nabla_{\gamma} \sigma - \sfP_{\beta \alpha} \sigma)$, and substituting in \eqref{equation:codifferential-nabla} gives that $L_0$ takes the form\footnote{The upper index ${}^A$ here is just an index on the standard tractor bundle; we use it here only for consistency.}
\begin{equation}
L_0(\sigma)^A \stackrel{\theta}{=}
	\begin{pmatrix}
            \star \\
            \star \\
            -\mcL^{\gamma \zeta}(\nabla_{\gamma} \nabla_{\zeta} \sigma - \sfP_{\gamma \zeta} \sigma) \\
            \nabla_{\alpha} \sigma \\
            \sigma \\
	\end{pmatrix} .
\end{equation}
Finally, composing gives
\begin{equation}\label{equation:nabla-splitting-operator}
	\nabla^{\mathcal{V}}_{\beta} L_0(\sigma)^A
		\stackrel{\theta}{=}
			\begin{pmatrix}
				\star\\
				\star\\
				\star\\
				\nabla_{\beta}\nabla_{\alpha} \sigma - \sigma \sfP_{\beta \alpha} - \tfrac{1}{2} \mcL^{\gamma \zeta}(\nabla_{\gamma} \nabla_{\zeta} \sigma - \sigma \sfP_{\gamma \zeta}) \mcL_{\beta \alpha} \\
				0 \\
			\end{pmatrix} ,
\end{equation}
and the second-to-bottom component is precisely the image of $\nabla^{\mathcal{V}}_{\beta} L_0(\sigma)$ under the projection $\ker \partial^* \to \ker \partial^* / \im \partial^* \cong \smash{\bigodot^2 \mbD^* [1]}$.\footnote{In fact, we can perform this computation with a little less effort: Since we project $\nabla^{\mcV} L_0(\sigma)$ onto the space $\mbD^* [1]$ of (appropriately weighted) symmetric tensors, we only need to compute the symmetric part of the second-lowest slot of the right-hand side in the first formula in Proposition \ref{proposition:splitting}(3), namely, $\smash{\nabla_{(\alpha} \tau_{\beta)} - \sigma \sfP_{(\alpha \beta)}}$. In particular, this does not involve $\upsilon$, so we can still derive the formula for $\Theta_0$ without computing the middle slots of the right-hand sides of \eqref{equation:codifferential-nabla} and \eqref{equation:nabla-splitting-operator}.}

\begin{proposition}\label{proposition:BGG}
The first BGG operator $\Theta_0:\Gamma(\mathcal{E}[1])\to \smash{\Gamma(\bigodot^2\mbD^*[1])}$ corresponding to the standard representation (the unique $7$-dimensional representation) $\bbV$ of $\G_2$ is
\begin{equation}\label{equation:BGG}
	\Theta_0(\sigma)_{\alpha \beta} = \nabla_{(\alpha}\nabla_{\beta)}\sigma-\sfP_{(\alpha \beta)}\sigma .
\end{equation}
\end{proposition}

We will see in the discussion after Proposition \ref{proposition:Levi-Civita} that the lowest-homogeneity component $\sfP_{\alpha \beta}$ of the Rho tensor determined by a scale $\sigma$ is already symmetric, that is, that $\sfP_{(\alpha \beta)}= \sfP_{\alpha \beta}$.

It is also possible to derive \eqref{equation:BGG} naively, that is, without computing the tractor covariant derivative $\nabla^{\mcV}$ and splitting operator $L_0$.

\begin{remark}[Alternative derivation of \eqref{equation:BGG}]
Since $\Theta_0$ is an operator $\smash{\Gamma(\mcE[1]) \to \Gamma(\bigodot^2 \mbD^*[1])}$, it has leading term $\sigma \mapsto \nabla_{(\alpha} \nabla_{\beta)} \sigma$, and since $\Theta_0$ is invariant, the leading term determines the operator modulo invariant curvature terms. The leading term is not invariant, but instead under a change of Weyl connection $s \rightsquigarrow \hat s = s \exp \Upsilon_1 \exp \Upsilon_2 \exp \Upsilon_3$ transforms as $\sigma \mapsto \nabla_{(\alpha} \nabla_{\beta)} \sigma + [\nabla_{(\alpha} (\Upsilon_1)_{\beta)} - \tfrac{1}{2} (\Upsilon_1)_{\alpha} (\Upsilon_1)_{\beta}] \sigma$ (see the general formula \cite[Proposition 5.1.6]{CapSlovak}). On the other hand, specializing \cite[Proposition 5.1.8]{CapSlovak} gives that $\sfP_{(\alpha \beta)}$ transforms as $\smash{\widehat \sfP_{(\alpha \beta)}} = \sfP_{(\alpha\beta)} + \nabla_{(\alpha} (\Upsilon_1)_{\beta)} - \tfrac{1}{2} (\Upsilon_1)_{\alpha} (\Upsilon_1)_{\beta}$. Thus, $\sigma \mapsto \nabla_{(\alpha} \nabla_{\beta)} \sigma - \sfP_{(\alpha \beta)} \sigma$ is invariant under changes of the Weyl structure. On the other hand, any invariant differential operator $\smash{\Gamma(\mcE[1]) \to \Gamma(\bigodot^2 \mbD^*[1])}$ has homogeneity $+2$, whereas the harmonic curvature has homogeneity $+4$, so no invariant curvature terms can appear in such an operator.
\end{remark}

The explicit formula \eqref{equation:BGG} furnishes another interpretation of solutions of $\Theta_0$ (or at least, nonvanishing solutions), in particular one that makes no reference to the induced conformal structure.

\begin{corollary}\label{corollary:235-characterization}
For an oriented $(2, 3, 5)$ distribution, a scale $\sigma$ (that is, a nonvanishing section $\sigma \in \Gamma(\mcE[1])$) is a solution of $\ker \Theta_0$ iff the lowest-homogeneity component $\sfP_{\alpha \beta}$ of the Rho tensor $\sfP$ in that scale is zero.
\end{corollary}
\begin{proof}
In the scale $\sigma$, the partial connection $\nabla$ satisfies $\nabla \sigma = 0$ and $\sfP_{\alpha \beta}$ is symmetric, so $\Theta_0(\sigma)_{\alpha \beta} = -\sfP_{\alpha \beta} \sigma$.
\end{proof}

\begin{remark}
One can formulate the analogous problem for the geometry of (oriented) $(3, 6)$ distributions---$6$-manifolds $N$ equipped with an oriented $3$-plane distribution $\mbE$ satisfying $[\mbE, \mbE] = TN$---which shares numerous features with that of oriented $(2, 3, 5)$ distributions. For our purposes the most important are that (1) this geometry can be realised as a parabolic geometry \cite[Subsubsection 4.3.2]{CapSlovak}, and (2) any such distribution determines a canonical conformal structure $\mbc_{\mbE}$ on $N$, this time owing to the exceptional inclusion $\Spin_0(3, 4) \hookrightarrow \SO(4, 4)$ \cite{Bryant36}. For such conformal structures, we may identify $\mcE[1]$ with the positive square root of $\smash{\bigwedge^3 \mbE}$, and the BGG operator corresponding to the $8$-dimensional spin representation of $\Spin_0(3, 4)$ is the map $\Theta_0 : \Gamma(\mcE[1]) \to \smash{\Gamma(\bigodot^2 \mbE^* [1])}$ given by applying the conformal almost Einstein operator $\wt\Theta_0$ and pulling back by the inclusion $\mbE \hookrightarrow TN$. One can proceed as we have done here for oriented $(2, 3, 5)$ distributions to check whether solutions $\Theta_0$ are automatically normal, and hence whether we may interpret the kernel of that operator as the space of almost Einstein scales of $\mbc_{\mbE}$. The analysis of the normalization condition mentioned in the proof of Theorem \ref{theorem:normality-of-solutions} for the geometry of oriented $(3, 6)$ distributions would be somewhat more algebraically involved than the analysis in that proof, as the bundle where harmonic curvature takes values for that geometry, namely $\smash{(\bigodot^2 \mbE \otimes \bigodot^2 \mbE^*)_{\circ} [2]}$, is comparatively complicated.
\end{remark}

\section{Relationship with the induced conformal geometry}
\label{section:conformal-structure}
In \cite{Nurowski} Nurowski showed that a $(2,3,5)$ distribution $\mbD$ on a $5$-manifold $M$ canonically determines a conformal structure $\mathbf{c}_{\mbD}$ of signature $(2,3)$ on $M$.

If $\mbD$ is oriented, then taking the determinant of the isomorphism
\[
	TM \cong (TM / [\mbD, \mbD]) \fplusl ([\mbD, \mbD] / \mbD) \fplusl \mbD
\]
and invoking the Levi bracket isomorphisms  gives
\begin{align*}
	{\textstyle \bigwedge^5 TM}
		&\cong
			\smash{
			{\textstyle \bigwedge^2 (          TM / [\mbD, \mbD])} 
				\otimes
			                        ([\mbD, \mbD] / \mbD        )
				\otimes
			{\textstyle \bigwedge^2          \mbD                }
			}\\
		&\cong
			\smash{
			{\textstyle \bigwedge^2 (\mbD \otimes \bigwedge^2 \mbD)}
				\otimes
			{\textstyle \bigwedge^2  \mbD                          }
				\otimes
			{\textstyle \bigwedge^2  \mbD                          }
			}\\
		&\cong
			\smash{
			({\textstyle \bigwedge^2 \mbD})^{\otimes 5}
			} .
\end{align*}
So, we may canonically identify $\smash{\bigwedge^2 \mbD}$ with a $5$th root of $\bigwedge^5 TM$; this root is the bundle of conformal scales of weight $1$ and in that context is usually denoted $\mcE[1]$ (see, for example, \cite{BEG}), which motivated our use of the notation $\mcE[1]$ for $\smash{\bigwedge^2 \mbD}$.

A scale $\theta\in\Gamma(\gr_{-2}(TM))=\Gamma(\mathcal{E}[1])$ determines a representative  metric $g_{\theta}\in\mathbf{c}_{\mbD},$ which in the index notation introduced earlier is  given by 
\begin{equation}\label{metric}
g_{\theta}(\xi,\eta) = (\xi_{-1})^{\alpha}(\eta_{-3})_{\alpha}-\xi_{-2} \eta_{-2}+(\xi_{-3})_{\alpha}(\eta_{-1})^{\alpha}
\end{equation}
for $\xi\stackrel{\theta}{=}(\xi_{-3},\xi_{-2},\xi_{-1})$, $\eta \stackrel{\theta}{=}(\eta_{-3},\eta_{-2},\eta_{-1})$.

\begin{proposition}\label{proposition:Levi-Civita} Let $\theta$ be a scale of an oriented $(2, 3, 5)$ distribution, and let $g_{\theta}$ denote the associated metric.
\begin{enumerate}
\item In the splitting determined by $\theta$, the Weyl connection $\nabla$ determined by $\theta$ is related to the Levi-Civita connection $\wt\nabla$ of  $g_{\theta}$ as follows:
\[\wt\nabla_{\beta}
    \left(
        \begin{array}{c}
            (\eta_{-1})^{\alpha} \\
            \eta_{-2} \\
            (\eta_{-3})_{\alpha}\\
                  \end{array}
    \right) =
     \left(
        \begin{array}{c}
           \nabla_{\beta}(\eta_{-1})^{\alpha}+4\eta_{-2} \mathcal{L}^{\gamma \alpha} \sfP_{\beta \gamma}-2 \sfP_{\beta\blacklozenge} \mathcal{L}^{\alpha \gamma}  (\eta_{-3})_{\gamma}\\
            \nabla_{\beta}\eta_{-2} + \tfrac{1}{2} \mathcal{L}_{\beta \gamma}(\eta_{-1})^{\gamma}+4 \mathcal{L}^{\gamma \zeta} \sfP_{\beta \gamma}  (\eta_{-3})_{\zeta}  \\
            \nabla_{\beta} (\eta_{-3})_{\alpha}+\tfrac{1}{2}\eta_{-2} \mathcal{L}_{\beta\alpha}\\
                    \end{array}
    \right)
\]

\[\wt\nabla_{\blacklozenge}
    \left(
    \begin{array}{c}
                      (\eta_{-1})^{\alpha} \\
            \eta_{-2} \\
            (\eta_{-3})_{\alpha}\\
        \end{array}
            \right) =
     \left(
                \begin{array}{c}
                \nabla_{\blacklozenge} (\eta_{-1})^{\alpha} +4 \eta_{-2} \mathcal{L}^{\gamma \alpha} \sfP_{\blacklozenge\gamma}-2 \sfP_{\blacklozenge\blacklozenge} \mathcal{L}^{\alpha \gamma}(\eta_{-3})_{\gamma}\\
          \nabla_{\blacklozenge}  \eta_{-2} +4 \mathcal{L}^{\gamma \alpha} \sfP_{\blacklozenge\gamma} (\eta_{-3})_{\alpha} \\
         \nabla_{\blacklozenge}   (\eta_{-3})_{\alpha} -\frac{1}{4} \mathcal{L}_{\beta \alpha} (\eta_{-1})^{\beta}\\
                    \end{array}
        \right)
\]

\[\wt\nabla^{\beta}
    \left(
        \begin{array}{c}
            (\eta_{-1})^{\alpha} \\
            \eta_{-2} \\
            (\eta_{-3})_{\alpha} \\
        \end{array}
    \right) =
     \left(
        \begin{array}{c}
             \nabla^{\beta}   (\eta_{-1})^{\alpha}  +4 \eta_{-2} \mathcal{L}^{\gamma \alpha} {\sfP^{\beta}}_{\gamma}-2\sfP^{\beta}{}_{\blacklozenge} \mathcal{L}^{\alpha \gamma}(\eta_{-3})_{\gamma}\\
            \nabla^{\beta} \eta_{-2}+4 \mathcal{L}^{\gamma \alpha} {\sfP^{\beta}}_{\gamma} (\eta_{-3})_{\alpha}  \\
          \nabla^{\beta}  (\eta_{-3})_{\alpha} \\
               \end{array}
    \right)
\]
\item The Rho tensor  $\sfP\stackrel{\theta}{=}(\sfP_{1},\sfP_2,\sfP_3)\in\Omega^1(M,T^*M)$ determined by $\theta$ and the conformal Schouten tensor $\wt\sfP\stackrel{\theta}{=}(\wt\sfP_{1},\wt\sfP_2,\wt\sfP_3)\in\Omega^1(M,T^*M)$ of $g_{\theta}$ are related by
\[\sfP_1=-\wt\sfP_1,\quad\sfP_2=-\tfrac{1}{2}\wt\sfP_2,\quad \sfP_3=-\tfrac{1}{3}\wt\sfP_3.\footnote{The negative here is a consequence of the fact that the usual convention for the conformal Schouten tensor (see, e.g., \cite{BEG}) differs in sign from the specialization to conformal geometry of the general definition of Rho tensor in \cite{CapSlovak}.}\]
\end{enumerate}
\end{proposition}
\begin{proof}
The tractor connection $\nabla^{\mathcal{V}}$ induced by the regular, normal Cartan geometry determined by the $(2,3,5)$ distribution $\mbD$ coincides with the normal conformal tractor connection of the associated conformal structure $\mathbf{c}_{\mbD}$, see \cite[Proposition 4]{HammerlSagerschnig}. The result then follows from comparing the formulae for the two tractor connections in the splitting determined by the scale $\theta$. 
\end{proof}

Part (2) of Proposition \ref{proposition:Levi-Civita} explains the choice of coefficients in \eqref{equation:Rho-components}. Let $\iota^a{}_{\alpha}$, $\iota^a{}_{\blacklozenge}$, $\iota^{a \bar\beta}$ denote the inclusions into $TM$ of the subspaces respectively identified with the associated graded bundles $\gr_{-1}(TM)$, $\gr_{-2}(TM)$, $\gr_{-3}(TM)$ by the scale $\theta$. Then, the quantities $\smash{\sfP_{\alpha \beta}, \sfP_{\alpha \blacklozenge}, \sfP_{\blacklozenge \alpha}, \sfP_{\alpha}{}^{\bar\beta}, \sfP_{\blacklozenge\blacklozenge}, \sfP^{\bar\alpha}{}_{\beta}, \sfP_{\blacklozenge}{}^{\bar\beta}, \sfP^{\bar\alpha}{}_{\blacklozenge}, \sfP^{\bar\alpha \bar\beta}}$ coincide with (the negatives of) the corresponding restrictions of $\wt{\sfP}$:
\begin{alignat*}{6}
	\sfP_{\alpha \beta} 
		&= - \iota^a{}_{\alpha} \iota^b{}_{\beta} \wt{\sfP}_{ab} , &\qquad
	\sfP_{\alpha \blacklozenge} = \sfP_{\blacklozenge \alpha}
		&= - \iota^a{}_{\alpha} \iota^b{}_{\blacklozenge} \wt{\sfP}_{ab} , &\qquad
	\sfP_{\alpha \bar\beta} = \sfP^{\bar\beta}{}_{\alpha}
		&= - \iota^a{}_{\alpha} \iota^{b \bar\beta} \wt{\sfP}_{ab} , \\
	\sfP_{\blacklozenge\blacklozenge}
		&= -\iota^a{}_{\blacklozenge} \iota^b{}_{\blacklozenge} \wt{\sfP}_{ab} , &\qquad
	\sfP_{\blacklozenge}{}^{\bar\alpha} = \sfP^{\bar\alpha}{}_{\blacklozenge}
		&= -\iota^a{}_{\blacklozenge} \iota^{b \bar\alpha} \wt{\sfP}_{ab} , &\qquad
	\sfP^{\bar\alpha \bar\beta}
		&= -\iota^{a \bar\alpha} \iota^{b \bar\beta} \wt{\sfP}_{ab} .
\end{alignat*}
The symmetry $\wt{\sfP}_{ab} = \wt{\sfP}_{ba}$ of the conformal Schouten tensor implies the identities
$\sfP_{\alpha            \blacklozenge} = \sfP_{\blacklozenge     \alpha      }$,
$\sfP_{\alpha}       {}^{\bar\beta    } = \sfP^{\bar\beta    }{}_{\alpha      }$,
$\sfP_{\blacklozenge}{}^{\bar\alpha   } = \sfP^{\bar\alpha   }{}_{\blacklozenge}$
recorded above as well as the symmetry $\sfP_{\alpha \beta} = \sfP_{\beta \alpha}$ mentioned after Proposition \ref{proposition:BGG} and the symmetry $\smash{\sfP^{\bar\alpha \bar\beta} = \sfP^{\bar\beta \bar\alpha}}$. The first and third of the former identities together show that one could, without ambiguity, suppress the lozenge indices in these symbols, that is, simply denote $\sfP_{\alpha} := \sfP_{\alpha \blacklozenge} = \sfP_{\blacklozenge \alpha}$, $\sfP := \sfP_{\blacklozenge\blacklozenge}$, $\smash{\sfP^{\bar\alpha} := \sfP_{\blacklozenge}{}^{\bar\alpha} = \sfP^{\bar\alpha}{}_{\blacklozenge}}$.

We are now prepared to prove the conformal formula for $\Theta_0$ given in the introduction.

\begin{proof}[Proof of Proposition \ref{proposition:pullback}]
Since $\mbD$ is totally isotropic with respect the induced conformal structure $\mbc_{\mbD}$ \cite{HammerlSagerschnig}, the restriction of the trace part of $\wt\Theta_0(\sigma)$ to $\mbD$ is zero, and hence (consulting \eqref{equation:BGG-conformal}) $\iota^* \wt\Theta_0(\sigma) = \iota^*(\wt\nabla^2 \sigma + \wt\sfP \sigma)$; note the quantity in parentheses is $\wt\nabla^2 \sigma + \wt\sfP \sigma$, not just its tracefree part. By the previous paragraph, we have $(\iota^* \wt\sfP)_{\alpha\beta} = -\sfP_{\alpha\beta}$, so comparing this formula for $\iota^*\wt\Theta_0$ with the formula \eqref{equation:BGG} for $\Theta_0$ shows that the claim is equivalent to a particular compatibility between the Hessians of $\nabla$ and $\wt\nabla$, namely that $\iota^a{}_{\alpha} \iota^b{}_{\beta} \wt\nabla_a \wt\nabla_b \sigma = \nabla_{(\alpha} \nabla_{\beta)} \sigma$. Using Proposition \ref{proposition:Levi-Civita}(1) to compute expressions for components of the connection dual to $\wt\nabla$ gives that
\[
	\iota^a{}_{\alpha} \iota^b{}_{\beta} \wt\nabla_a \wt\nabla_b \sigma + \tfrac{1}{2} \iota^c{}_{\blacklozenge} \wt\nabla_c \sigma \mcL_{\alpha \beta}
		= \nabla_{\alpha} \nabla_{\beta} \sigma .
\]
Since $\mcL_{\alpha \beta}$ is skew, symmetrizing yields the desired compatibility.
\end{proof}

\begin{remark}[Induced conformal Killing fields]
Note that the two lowest-homogeneity components of $\mcV$, namely $\mcE[1]$ and $\gr_{-3}(TM) [-1] = (TM / [\mbD, \mbD]) [-1] \cong \mbD$ are the second- and third-lowest--homogeneity components of $TM$, equivalently, the components of $[\mbD, \mbD]$. This suggests that one can build a natural, first-order differential operator $\iota_7: \Gamma(\mcE[1]) \to \Gamma([\mcD, \mbD]) \subset \Gamma(TM)$ by truncating $L_0$ after the second-lowest--homogeneity component and interpreting the result as a vector field in $[\mbD, \mbD]$:
\[
	\iota_7 :
		\sigma
			\mapsto
		\left(
			\begin{array}{c}
				\mathcal{L}^{\beta \alpha} \nabla_{\beta} \sigma \\
				\sigma \\
				0
			\end{array}
		\right)
			\in
		\Gamma
		\left(
			\begin{array}{c}
				\mbD \\%
				{}[\mbD, \mbD] / \mbD \\
				TM / [\mbD, \mbD]
			\end{array}
		\right)
		\cong
		\Gamma(TM) .
\]
One can verify that this operator is indeed well-defined by computing its transformation under a change of scale, but one can also derive it in a manifestly invariant way from the $\mfg_2$-module inclusion $\bbV \hookrightarrow \mfg_2 \oplus \bbV \cong \mfso(3, 4)$ using BGG splitting and projection operators \cite[Theorem B]{HammerlSagerschnig}. If $\sigma$ is a nonzero almost Einstein scale, then $\iota_7(\sigma)$ is a conformal Killing field of $\mbc_{\mbD}$ but its flow does not preserve the distribution $\mbD$, and all conformal Killing fields of $\mbc_{\mbD}$ that are sections of $[\mbD, \mbD]$ arise this way. In particular, flowing the distribution $\mbD$ along $\iota_7(\sigma)$ yields a $1$-parameter family of distinct distributions that all induce the same conformal structure $\mbD$, relating the almost Einstein scales to the so-called conformal isometry problem for $(2, 3, 5)$ distributions; see \cite[\S\S3, 4]{SagerschnigWillse} for much more.
\end{remark}

\section{Monge normal form}
\label{section:Monge-normal-form}

One way to construct $(2, 3, 5)$ distributions is via ordinary differential equations of the form
\begin{equation}\label{equation:ode}
	z'(x) = F(x, y(x), y'(x), y''(x), z(x)) :
\end{equation}
We may prolong any solution $(x, y(x), z(x))$ to a curve $(x, y(x), y'(x), y''(x), z(x))$ in the jet space $J := J^{2, 0}(\bbR, \bbR^2) \cong \bbR^5$, and by construction this curve will be tangent to the $2$-plane distribution $\mbD_F \subset TJ$ defined in respective coordinates $(x, y, p, q, z)$ as the common kernel of the canonical jet $1$-forms $dy - p \,dx$ and $dp - q \,dx$ and the $1$-form $dz - F(x, y, p, q, z) dx$. Conversely, the projection of any integral curve of this distribution (to which the pullback of $dx$ is nowhere zero) to $xyz$-space defines a solution of the o.d.e. \eqref{equation:ode}.

The distribution $\mbD_F$ is spanned by the coordinate vector field
\[
	\mbQ := \partial_q
\]
and the total derivative
\[
	\mbX := \partial_x + p \partial_y + q \partial_p + F(x, y, p, q, z) \partial_z ,
\] and computing directly shows that $\mbD_F$ is a $(2, 3, 5)$ distribution iff $\mbQ^2 F = F_{qq}$ is nowhere zero.
In fact, every $(2, 3, 5)$ distribution locally arises this way, and hence we say that a $(2, 3, 5)$ distribution $\mbD_F$ is in \textit{Monge normal form}.

\begin{proposition}\cite[\S76]{Goursat}
Let $(M, \mbD)$ be a $(2, 3, 5)$ distribution and fix a point $u \in M$. There is a neighborhood $U \subset M$ of $u$, a diffeomorphism $h: U \to h(U) \subset J$ for which $h(u) = 0$, and a smooth function $F : h(U) \to \Bbb R$ for which $Th \cdot \mbD\vert_U = \mbD_F$.
\end{proposition}

We record some formulae here for a distribution $\mbD_F$ in Monge normal form. In all cases we compute with respect to the scale $\mbQ \wedge \mbX \in \smash{\Gamma(\bigwedge^2 \mbD)} = \Gamma(\mcE[1])$, for which many expressions simplify.

The generalised Reeb field $R$ is
\[
	R = [\mbQ, \mbX] + \left(\frac{\mbX \mbQ^2 F}{\mbQ^2 F} - F_z\right) \mbQ - \frac{\mbQ^3 F}{\mbQ^2 F} \mbX .
\]
The partial connection $\nabla : \Gamma(\mbD) \times \Gamma(\mbD) \to \Gamma(\mbD)$ is characterised by
\begin{alignat*}{3}
	              \nabla_{\mbQ} \mbQ &= 0
	&\qquad\qquad \nabla_{\mbX} \mbQ &= \left(\frac{\mbX \mbQ^2 F}{\mbQ^2 F} - F_z\right) \mbQ - \frac{\mbQ^3 F}{\mbQ^2 F} \mbX \\
	              \nabla_{\mbQ} \mbX &= 0
	&             \nabla_{\mbX} \mbX &= \frac{(\mbQ \mbX^2 - 3 \mbX [\mbQ, \mbX]) F}{\mbQ^2 F} \mbQ - \left(\frac{\mbX \mbQ^2 F}{\mbQ^2 F} - F_z\right) \mbX ,
\end{alignat*}
and the dual partial connection $\nabla : \Gamma(\mbD) \times \Gamma(\mbD^*) \to \Gamma(\mbD^*)$ is characterised (in the coframe $(\eta_{\mbQ}, \eta_{\mbX})$ of $\mbD$ dual to $(\mbQ, \mbX)$) by
\begin{alignat*}{3}
	              \nabla_{\mbQ} \eta_{\mbQ} &= 0
	&\qquad\qquad \nabla_{\mbX} \eta_{\mbQ} &= -\left(\frac{\mbX\mbQ^2 F}{\mbQ^2 F} - F_z\right) \eta_{\mbQ} - \frac{(\mbQ \mbX^2 - 3 \mbX [\mbQ, \mbX]) F}{\mbQ^2 F} \eta_{\mbX} \\
	              \nabla_{\mbQ} \eta_{\mbX} &= 0
	&             \nabla_{\mbX} \eta_{\mbX} &= \frac{\mbQ^3 F}{\mbQ^2 F} \eta_{\mbQ} + \left(\frac{\mbX \mbQ^2 F}{\mbQ^2 F} - F_z\right) \eta_{\mbX} .
\end{alignat*}

Using \eqref{RhoT}
we can now compute the first BGG operator $\Theta_0$ with respect to the frame $(\mbQ, \mbX)$. For any $(2, 3, 5)$ distribution in Monge normal form and any section $\sigma \in \Gamma(\mcE[1])$, let $\ul\sigma$ be its trivialization with respect to the scale $\mbQ \wedge \mbX$, that is, the unique function satisfying $\sigma = \ul\sigma \mbQ \wedge \mbX$.


\begin{proposition}
Let $F$ be a function on an open subset on $J$ for which $F_{qq}$ vanishes nowhere and $\mbD_F$ the distribution it determines. The components of $\Theta_0$ with respect to the frame $(\mbQ, \mbX)$ of $\mbD$ are:
\begin{align}
	\Theta_0(\sigma)(\mbQ, \mbQ)
		&= \bigg[\underbrace{\mbQ^2 \ul\sigma}_{(\Sym \nabla^2 \ul\sigma)(\mbQ, \mbQ)} - \underbrace{\frac{1}{10} \left(-\frac{3 \mbQ^4 F}{\mbQ^2 F} + \frac{4 (\mbQ^3 F)^2}{(\mbQ^2 F)^2} \right)}_{\sfP(\mbQ, \mbQ)} \ul\sigma \bigg] \otimes (\mbQ \wedge \mbX) \label{equation:DQQ} \\
	\Theta_0(\sigma)(\mbQ, \mbX)
		&= \bigg[\underbrace{\frac{1}{2}\bigg(
				\mbQ \mbX
				+ \mbX \mbQ
				- \left(\frac{\mbX \mbQ^2 F}{\mbQ^2 F} - F_z\right) \mbQ
				+ \frac{\mbQ^3 F}{\mbQ^2 F} \mbX
				\bigg) \ul\sigma}_{(\Sym \nabla^2 \sigma)(\mbQ, \mbX)} \label{equation:DQX} \\
		&\phantom{\bigg[}\qquad - \frac{1}{10} \left(\frac{(-2 \mbQ \mbX \mbQ^2 - \mbX \mbQ^3) F}{\mbQ^2 F}\right. \notag \\
		&\phantom{\bigg[\qquad-}\underbrace{\phantom{\frac{1}{10}\bigg(}\qquad\left.
			+ \frac{4 (\mbQ^3 F) (\mbX \mbQ^2 F)}{(\mbQ^2 F)^2} - \frac{(\mbQ^3 F) F_z}{\mbQ^2 F} + 2 \mbQ \cdot F_z \right)}_{\sfP(\mbQ, \mbX)} \ul\sigma \bigg] \otimes (\mbQ \wedge \mbX) \notag \\
	\Theta_0(\sigma)(\mbX, \mbX)
		&= \bigg[
			\underbrace{\left(\mbX^2 - \frac{(\mbQ \mbX^2 - 3 \mbX [\mbQ, \mbX]) F}{\mbQ^2 F} \mbQ + \left(\frac{\mbX \mbQ^2 F}{\mbQ^2 F} - F_z\right) \mbX\right) \ul\sigma}_{(\Sym \nabla^2 \sigma)(\mbX, \mbX)} \label{equation:DXX} \\
		&\phantom{\bigg[}\qquad - \frac{1}{10}\left(\frac{(-\mbQ^2 \mbX^2 + 3 \mbQ \mbX [\mbQ, \mbX] + 2 \mbX \mbQ^2 \mbX - 4 \mbX \mbQ \mbX \mbQ) F}{\mbQ^2 F}\right. \notag \\
		&\phantom{\bigg[\qquad-}\underbrace{\phantom{\frac{1}{10}\bigg(}\qquad\left.
				+ \frac{(\mbQ^3 F)(3 \mbQ \mbX^2 - 9 \mbX [\mbQ, \mbX]) F}{(\mbQ^2 F)^2}
				+ \frac{(\mbX \mbQ^2 F)^2}{(\mbQ^2 F)^2} - F_z^2\right)}_{\sfP(\mbX, \mbX)} \ul\sigma \bigg] \otimes (\mbQ \wedge \mbX) \notag .
\end{align}
\end{proposition}

\section{Examples}
\label{section:examples}

In this section, we give several examples and a nonexample.

For each example, we record an additional invariant associated to any $(2, 3, 5)$ distribution $M$. Since we may regard the harmonic curvature for a geometry of this type as a section $\smash{A \in \Gamma(\bigodot^4 \mbD^*)}$, that is, as a field of binary quartic forms (this was essentially observed in \cite{CartanFiveVariables}), we can regard at each point $u \in M$ the complexification $A_u \otimes \Bbb C$ as a homogeneous quartic polynomial on $\mbD_u \otimes \bbC$. If $A_u \neq 0$, $A_u \otimes \Bbb C$ has four roots in $\bbP(\mbD_u \otimes \bbC)$ counting multiplicity, and so the multiplicities of the roots determine a partition of $4$ called the \textit{root type} of $\mbD$ at $u$; if $A_u = 0$, by convention we say that $\mbD$ has root type $\infty$ at $u$. This is analogous to the well-known notion of Petrov type in $4$-dimensional Lorentzian conformal geometry. If the root type $r_u$ of $\mbD$ at $u$ is the same for all $u \in M$, we say that $\mbD$ has \textit{constant root type} $r_u$.

We give examples of almost Einstein $(2, 3, 5)$ distributions of constant root types $\infty$, $[4]$, $[3, 1]$, and $[2, 2]$. In particular, no example of root type $[3, 1]$ seems to have been known previously, and an example of root type $[2, 2]$ was given only recently \cite{SagerschnigWillse} (see Example \ref{example:rolling-distribution}). Whether an almost Einstein $(2, 3, 5)$ distributions can have at any point either of the remaining root types, namely $[2, 1, 1]$ and $[1, 1, 1, 1]$, remains open.

\begin{example}[The flat model]\label{example:flat-model}
For the flat $(2, 3, 5)$ distribution $(\G_2 / P, \Delta)$ \cite{Sagerschnig}, which, since the harmonic curvature vanishes identically, has constant root type $\infty$, the standard tractor bundle is $\mcV = \G_2 \times_P \bbV \to \G_2 / P$, the holonomy of $\nabla^{\mcV}$ is trivial, and the parallel sections of $\mcV$ are the sections $t_v : gP \to \llbracket g, g^{-1} \cdot v \rrbracket$, $v \in \bbV$. Thus, (1) the space of almost Einstein scales of $\Delta$ is $\ker \Theta_0 = \{\Pi_0(t_v) : v \in \Bbb V\} \cong \bbV$ and has dimension $7$, and (2) this is the maximal possible dimension of $\ker \Theta_0$ for any $(2, 3, 5)$ distribution. Conversely, if for some $(2, 3, 5)$ distribution we have $\dim \ker \Theta_0 = 7$, then that distribution is flat.\footnote{The next-largest realizable value for $\dim \ker \Theta_0$ is unknown. Example \ref{example:Fq} shows that it is at least $2$, and it follows from the fact that the largest proper subalgebra of the split octonions has dimension $6$ (any such algebra is isomorphic to the sextonion algebra) \cite{Westbury} that it is no larger than $5$.}

A simple Monge normal form for $\Delta$ is $\mbD_F$, $F(x, y, p, q, z) = q^2$, for which the BGG equation $\Theta_0(\sigma) = 0$ simplifies dramatically to $\mbQ^2 \ul\sigma = \frac{1}{2} (\mbQ \mbX + \mbX \mbQ) \ul\sigma = \mbX^2 \ul\sigma = 0$. Solving gives $\ker \Theta_0$ in jet space coordinates:
\[
	\ker \Theta_0 = \langle
				2 x p q - 6 y q + 4 p^2 - 3 x z,
				2 p q - 3 z,
				x^2 q - 4 x p + 6 y,
				x q - 2 p,
				q,
				x,
				1
			 \rangle
			 \otimes (\mbQ \wedge \mbX) .
\]

\end{example}

\begin{example}[Distributions $\mbD_{F(q)}$]\label{example:Fq} The conformal structures induced by the Monge normal form distributions $\mbD_{F(q)}$ were essentially observed to be almost Einstein in \cite{Nurowski}. We can recover this result quickly: For a function $F(q)$, substituting in \eqref{equation:DQX} and \eqref{equation:DXX} gives $\Theta_0(\sigma)(\mbQ, \mbX) = \frac{1}{2} (\mbQ \mbX + \mbX \mbQ) \ul\sigma (\mbQ \wedge \mbX)$ and $\Theta_0(\sigma)(\mbX, \mbX) = \mbX^2 \ul\sigma (\mbQ \wedge \mbX)$. So, for the ansatz $\sigma = \sigma(q)$, the only nonzero component of $\Theta_0(\sigma) = 0$ is the equation $\Theta_0(\sigma)(\mbQ, \mbQ) = 0$, which we may regard as a second-order o.d.e. in $\ul\sigma(q)$:
\[
	\ul\sigma_{qq} - \frac{1}{10} \left(-\frac{3 F_{qqqq}}{F_{qq}} + \frac{4 F_{qqq}^2}{F_{qq}^2} \right) \ul\sigma = 0 .
\]
Thus, $\mbD_{F(q)}$ admits (at least) two linearly independent almost Einstein scales $\ul\sigma \mbQ \wedge \mbX$.

There is a fourth-order differential operator $\Xi$ such that the harmonic curvature of $\mbD_{F(q)}$ is $A = \Xi[\mbQ^2 F] \eta_{\mbQ}^4$. In particular, the root type of $\mbD_{F(q)}$ at $(x, y, p, q, z)$ is $\infty$ and $[4]$ respectively where $\Xi[\mbQ^2 F]$ is zero and nonzero. If $\mbD_{F(q)}$ is not flat, so that the root type is not $\infty$ everywhere, then the holonomy of the tractor connection $\nabla^{\mcV}$ is isomorphic to the Heisenberg $5$-group \cite{Willse}.

\end{example}

\begin{example}[A distribution of root type $\mathrm{[3, 1]}$]\label{example:31}
We construct an example by regarding the BGG equation $\Theta_0(\sigma) = 0$ as a system in the pair $(F, \sigma)$, choosing a manageable ansatz for $F$, and analyzing the resulting system enough to produce an explicit solution. In particular, the example we produce will have constant root type $[3, 1]$; to the knowledge of the authors no example of an almost Einstein $(2, 3, 5)$ distribution with root type $[3, 1]$ at any point is reported in the literature.

Substituting the ansatz $F(x, y, p, q, z) = q^2 + f(x, y, p, z)$ in the equation $\Theta_0(\sigma)(\mbQ, \mbQ) = 0$ gives that for any solution $\sigma = \ul\sigma \mbQ \wedge \mbX$, $\ul\sigma$ is affine in $q$, and substituting into $\Theta_0(\sigma)(\mbQ, \mbX) = 0$ and comparing terms in $q$ and then doing the same for $\Theta_0(\sigma)(\mbX, \mbX) = 0$ gives that $\ul\sigma$ is also affine in $z$. Now requiring that $\sigma$ is nonzero imposes strong restrictions on solutions $(F, \sigma)$, including that $f$ must also be affine in $z$, which turns out to imply that the harmonic curvature $A$ of $\mbD_F$ has a triple root at each point, or equivalently that the root type of $\mbD_F$ at each point is one of $\infty, [4], [3, 1]$. By further specializing the ansatz, one can find many solutions $(F, \sigma)$ of the system with $\sigma \neq 0$. For concreteness, we give a relatively simple example:
\[
	F^*(x, y, p, q, z) = q^2 - \frac{p^4}{(x + y)^2} - \frac{14 p^3}{3 (x + y)^2} + \frac{2 p z}{x + y}, \qquad \sigma^*(x, y, p, q, z) = (x + y) \mbQ \wedge \mbX .
\]
Computing gives that $A(\mbQ, \mbQ, \,\cdot\, , \,\cdot\,) = 0$ and that $A(\mbQ, \mbX, \mbX, \mbX)$ vanishes nowhere, from which we can conclude that $\mbD_{F^*}$ has constant root type $[3, 1]$.

The solution $\sigma^*$ spans $\ker \Theta_0$, that is, all of the Einstein representatives of the conformal class induced by $\mbD_{F^*}$ (in fact, they are Ricci-flat) are homothetic. Equivalently, the only parallel sections of the standard tractor bundle $\mcV$ are the multiples of $L_0(\sigma^*)$, but $\nabla^{\mcV}$ also preserves an isotropic rank-$2$ bundle $E \supset L_0(\sigma^*)$ whose fibers $E_u$ are special in the sense that the restrictions thereto of the parallel tractor cross product defined by the $\G_2$-structure are zero. All of the parallel tractor objects can be derived from these (and the parallel $\G_2$-structure) in the sense that the holonomy group $\Hol_u(\nabla^{\mcV})$, $u \in M$, is exactly the stabiliser in $\G_2$ of the data $\{L_0(\sigma^*)_u, E_u\}$, which turns out to be a codimension-$1$ (and so $7$-dimensional) subgroup of a Borel subgroup of $\G_2$.
\end{example}

\begin{example}[A special rolling distribution]\label{example:rolling-distribution}
Given a pair of oriented surfaces $(\Sigma_a, g_a)$, $a = 1, 2$, the space $C$ of all oriented configurations of the two surfaces with a tangent point of contact is the total space of a circle bundle $C \to \Sigma_1 \times \Sigma_2$. A base point $v := (v_1, v_2)$ specifies the points of contact on $\Sigma_1, \Sigma_2$, and the fiber over $v$ consists of the oriented isometries $T_{v_1} \Sigma_1 \to T_{v_2} \Sigma_2$ and so encodes the possible relative angular positions of the surfaces around the point of tangency. A path in this configuration space represents a relative motion of the two surfaces, and the so-called no-slip, no-twist conditions impose three linearly independent linear conditions on each tangent space $T_u C$ and hence define a $2$-plane distribution $\mbD$ on $C$. If the images of the curvatures of $g_1, g_2$ have empty intersection, $\mbD$ is a $(2, 3, 5)$ distribution \cite{AnNurowski, BryantHsu}.

Consider the case in which the two surfaces are a sphere $S$ and a hyperbolic plane $H$, scaled so that the sum of the (constant) curvatures of the two spaces is zero. It was shown in \cite{SagerschnigWillse} that the resulting $(2, 3, 5)$ distribution $\mbD$ has constant root type $[2, 2]$ and is almost Einstein: The space of almost Einstein scales is $1$-dimensional and any (equivalently every) nonzero solution vanishes nowhere, that is, there is a single Einstein representative of $\mbc_{\mbD}$ up to homothety. Computing shows that it has negative scalar curvature and that the holonomy of $\nabla^{\mcV}$ is $\SU(1, 2)$.

We can quickly recover the first result here: If we take the metrics on the sphere and hyperbolic space respectively to be
\[
	g_S = d\zeta^2 + \sin^2 \zeta \,d\theta^2
		\qquad \textrm{and} \qquad
	g_H = dt^2 + \sinh^2 t \,d\beta^2 ,
\]
then following the computational procedure of \cite{AnNurowski} using the orthonormal coframes $(\partial_{\zeta}, \csc \zeta \,\partial_{\theta})$ and $(\partial_t, \csch t \,\partial_{\beta})$ yields a frame $(\xi, \eta)$ of $\mbD$, where
\begin{align*}
	\xi &:= \partial_{\zeta} + \cos \phi \,\partial_t + \csch t \sin \phi \,\partial_{\beta} - \coth t \sin \phi \, \partial_{\phi}, \\
	\eta &:= \csc \zeta \,\partial_{\theta} - \sin \phi \,\partial_t + \csch t \cos \phi \, \partial_{\beta} + (-\coth t \cos \phi + \cot \zeta) \partial_{\phi} ;
\end{align*}
here, $\phi$ is a standard fiber coordinate on $C \to S \times H$. Computing gives that in the scale $\xi \wedge \eta$ the lowest-homogeneity component $\sfP_{\alpha\beta}$ of the Rho tensor vanishes, so Corollary \ref{corollary:235-characterization} gives that $\xi \wedge \eta$ is itself an almost Einstein scale. 

The flat model $(\G_2 / P, \Delta)$ discussed in Example \ref{example:flat-model} famously can be realised as the rolling distribution for a pair of spheres with ratio $3 : 1$ of radii \cite{BorMontgomery, Zelenko}. By treating the ratio of curvatures as a parameter and analyzing $\Theta_0$ as a function of that parameter, one can show that, among the rolling distributions for pairs of constant scalar curvature (Riemannian) surfaces, the only ones that admit almost Einstein scales are (up to local equivalence) the flat model and $(C, \mbD)$, but we take up a proof of this elsewhere.
\end{example}

\subsection{Application: Conformal structures of holonomy $\G_2$}
If a normal conformal tractor connection $\nabla^{\mcV}$ of an oriented $(2, 3, 5)$ conformal structure $\mbc_{\mbD}$ has holonomy a proper subgroup of $\G_2$, then by \cite[Proof of Theorem 1]{LeistnerNurowski} it admits (1) a nonzero parallel standard tractor or (2) a parallel isotropic tractor $2$-plane field whose fibers are special in the sense described in Example \ref{example:31}. Condition (2) implies \cite[Theorem 1.2]{GrahamWillse} that the harmonic curvature $A$ of $\mbD$ has a multiple root at each point, which gives an efficient way, recorded in the following proposition, to use the operator $\Theta_0$ to show that particular $(2, 3, 5)$ conformal structures have holonomy equal to $\G_2$.

\begin{proposition}\label{proposition:holonomy-G2}
Let $(M, \mbD)$ be an oriented $(2, 3, 5)$ distribution such that (1) $\ker \Theta_0$ is trivial, and (2) the root type of $\mbD$ at some point in $M$ is $[1, 1, 1, 1]$. Then, the holonomy of the normal tractor connection $\nabla^{\mcV}$, equivalently the normal conformal holonomy of $\mbc_{\mbD}$, is isomorphic to the full group $\G_2$.
\end{proposition}

\begin{example}[A nonexample: A homogeneous conformal structure of holonomy $\G_2$]
As an example, consider the $5$-dimensional Lie algebra $\mfg$ with bracket
\[
	\begin{array}{c|ccccc}
		[\,\cdot\, , \,\cdot\,]    & e_1 & e_2 & e_3 & e_4 & e_5 \\
	\hline
		e_1 & \cdot & 3 e_2 & 2 e_3 & e_4 & e_4 + e_5 \\
		e_2 & & \cdot & \cdot & \cdot & \cdot \\
		e_3 & & & \cdot & \cdot & e_2 \\
		e_4 & & & & \cdot & e_3 \\
		e_5 & & & & & \cdot
	\end{array} \quad .
\]

For any Lie group $G$ with Lie algebra $\mfg$, consider the left-invariant distribution $\mbD \subset TG$ characterised via the identification $T_{\id} G \cong \mfg$ by $\mbD_{\id} = \langle e_1 + e_3, e_5 \rangle$. Computing directly shows that $\mbD$ is a $(2, 3, 5)$ distribution and that $\mbD$ has constant root type $[1, 1, 1, 1]$. It can also be checked that $\mbD$ can be realised locally in Monge normal form $\mbD_F$ for $F(x, y, p, q, z) = y + \exp q$ \cite{DoubrovGovorovUnpublished}, and we use this form to analyze the system $\Theta_0(\sigma) = 0$.

For this $F$, $\Theta_0(\sigma)(\mbQ, \mbQ) = (\ul\sigma_{q, q} - \smash{\frac{1}{10} \ul\sigma}) \mbQ \wedge \mbX$, so for any solution $\sigma = \ul\sigma \mbQ \wedge \mbX$, $\ul \sigma$ has the form $\smash{\lambda(x, y, p, z) \exp\left(q / \sqrt{10}\right)} + \smash{\mu(x, y, p, z) \exp\left(- q / \sqrt{10}\right)}$. Substituting this expression for $\sigma$ in the combination $(\partial_q^2 - \frac{1}{10}) \Theta_0(\sigma)(\mbQ, \mbX)$ and comparing like coefficients in $q$ shows that neither $\lambda$ nor $\mu$ depend on either $p$ or $z$, then substituting again in $\Theta_0(\sigma)(\mbQ, \mbX)$ and comparing like terms in $p, z$ shows that $\lambda$ and $\mu$ must be constant. Finally, substituting again for $\sigma$ in $\Theta_0(\sigma)(\mbX, \mbX)$ shows that $\sigma = 0$, so $\ker \Theta_0$ is trivial, that is, $\mbD$ is not almost Einstein. By Proposition \ref{proposition:holonomy-G2}, the holonomy of $\nabla^{\mcV}$, or equivalently the holonomy of the induced conformal structure $\mbc_{\mbD}$, is $\G_2$.
\end{example}

\appendix
\section{}
\label{algebra}
We use the following matrix representation of the Lie algebra $\mfg_2$ of $\G_2$:
\[
    \mfg_2 =
        \left\{
           \left(
                \begin{array}{ccccccc}
                    -\tr A & Z & s & W^{\top} & 0 \\
                        X & A - (\tr A) \bbI & \sqrt{2} \bbJ Z^{\top} & \frac{1}{\sqrt{2}} s \bbJ & -W \\
                        r & -\sqrt{2} X^{\top} \bbJ & 0 & -\sqrt{2} Z \bbJ & s \\
                      Y^{\top} & -\frac{1}{\sqrt{2}} r \bbJ & \sqrt{2} \bbJ X & (\tr A) \bbI - A^{\top} & -Z^{\top} \\
                        0 & -Y & r & -X^{\top} & \tr A   \\
                \end{array}
            \right)
            :
            \begin{array}{c}
           		X, W \in \bbR^2      ; \\
				r, s \in \bbR        ; \\
				Y, Z \in (\bbR^2)^*  ; \\
				A    \in \mfgl(2, \bbR)
           \end{array}
        \right\} \textrm{,}
\]
where
\[
	\bbI =
		\begin{pmatrix}
			1 & 0 \\
			0 & 1
		\end{pmatrix},
		\qquad
	\bbJ =
		\begin{pmatrix}
			0 & -1 \\
			1 &  0
		\end{pmatrix}.
\]

The subgroup $P$ stabilizing the ray spanned by the first basis element of $\mathbb{V}$ is parabolic.
The $|3|$-grading $(\mfg_a)$ of $\mfg_2$ is schematised by the block decomposition labeling
\begin{equation}\label{grading}
    \left(
        \begin{array}{ccccc}
            \mfg_0    & \mfg_{+1} & \mfg_{+2} & \mfg_{+3} &    0      \\
            \mfg_{-1} & \mfg_0    & \mfg_{+1} & \mfg_{+2} & \mfg_{+3} \\
            \mfg_{-2} & \mfg_{-1} &    0      & \mfg_{+1} & \mfg_{+2} \\
            \mfg_{-3} & \mfg_{-2} & \mfg_{-1} & \mfg_0    & \mfg_{+1} \\
               0      & \mfg_{-3} & \mfg_{-2} & \mfg_{-1} & \mfg_0    \\
        \end{array}
    \right) \textrm{.}
\end{equation}
The Lie algebra of $P$ is the direct sum $\mathfrak{p} = \mathfrak{g}_0 \oplus \mathfrak{g}_{+1} \oplus \mathfrak{g}_{+2} \oplus \mathfrak{g}_{+3}$. As usual, we denote $\mfg_- := \mfg_{-3} \oplus \mfg_{-2} \oplus \mfg_{-1}$ and $\mfg _+ := \mfg_{+1} \oplus \mfg_{+2} \oplus \mfg_{+3}$.

The Lie bracket component $\mfg_{-1}\times\mfg_{-1}\to\mfg_{-2}$ inducing the Levi bracket component $\mcL : \smash{\textstyle \bigwedge^2 \mbD} \to [\mbD, \mbD] / \mbD$ is given by $(X,X')\mapsto \smash{2\sqrt{2}\epsilon_{\alpha \beta} X^{\alpha} (X')^{\beta}}$, and the component $\mfg_{-1} \times \mfg_{-2} \to \mfg_{-3}$ inducing the component $\mcL : \mbD \otimes ([\mbD, \mbD] / \mbD) \to TM / [\mbD, \mbD]$ is given by $(X,r') \mapsto \smash{\tfrac{3}{\sqrt{2}}r'\epsilon_{\beta\alpha}X^{\beta}}$. Here, $\eps_{\alpha \beta}$ (and later, $\eps^{\alpha \beta}$) is the Levi-Civita symbol on $\Bbb R^2$.

We use $\tfrac{1}{24}K$, where $K$ is the Killing form, to identify $\mathfrak{g}_{i}\cong(\mathfrak{g}_{-i})^*$; the dualities are: 
\[
    \begin{array}{c|ccl}
        \mfg_{-3} \times \mfg_{+3} & (Y, W ) & \mapsto & \tfrac{1}{3} Y_{\alpha} W^{\alpha} \\
        \mfg_{-2} \times \mfg_{+2} & (r, s ) & \mapsto & \tfrac{1}{2} r s   \\
        \mfg_{-1} \times \mfg_{+1} & (X, Z ) & \mapsto & Z_{\alpha} X^{\alpha} \\
        \mfg_{ 0} \times \mfg_{ 0} & (A, A') & \mapsto & \frac{1}{3} [A^{\alpha}{}_{\beta} (A')^{\beta}{}_{\alpha} + A^{\alpha}{}_{\alpha} (A')^{\beta}{}_{\beta}]
    \end{array}
    \textrm{.}
\]

When restricted to $\GL(2, \bbR) < P$, the standard representation $\mathbb{V}$ decomposes as
\[
	\bbV
		=
			\left(
				\begin{array}{c}
					\bbV_{ 2} \\
					\bbV_{ 1} \\
					\bbV_{ 0} \\
					\bbV_{-1} \\
					\bbV_{-2}
				\end{array}
			\right)
		\cong
			\left(
				\begin{array}{c}
					\smash{\bigwedge^2 (\bbR^2)^*} \\
					                   (\bbR^2)^*  \\
					                    \bbR       \\
					                    \bbR^2     \\
					\smash{\bigwedge^2  \bbR^2   }
				\end{array}
			\right) ,
\]
and we write a generic element of $\bbV$ as
\[
    \left(
        \begin{array}{c}
            \chi \\
            \phi \\
            \upsilon \\
            \tau^{\top} \\
            \sigma \\
        \end{array}
    \right) \textrm{.}
   \]
The notation $\cdot^{\top}$ reflects that we view $\tau \in \bbR^2$ as an element of $\smash{\textstyle (\bbR^2)^* \otimes \bigwedge^2 \bbR^2}$ and thus as a row vector (and so decorate it with a lower index), and so view $\tau^{\top}$ as a column vector. Likewise, we view $\phi \in (\bbR^2)^*$ as an element of $\smash{\textstyle \bbR^2 \otimes \bigwedge^2 (\bbR^2)^*}$ and thus as a column vector (and so mark it with an upper index).

The action $\mfg_2 \times \bbV \to \bbV$ decomposes into the following maps $\mfg_a \times \bbV_b \to \bbV_{a + b}$:
\begin{equation}\label{table}
	\begin{array}{c|c|ccccc}
		&   & \bbV_{-2} & \bbV_{-1} & \bbV_0 & \bbV_1 & \bbV_2 \\
		\hline
		&   & \sigma & \tau & \upsilon & \phi & \chi \\
		\hline
		\mfg_{-3} & Y
			& 0
			& 0
			& 0
			& - Y_{\alpha} \phi^{\alpha}
			& \chi Y_{\alpha} \\
		\mfg_{-2} & r
			& 0
			& 0
			& r \upsilon
			& -\frac{1}{\sqrt{2}} r \eps_{\beta\alpha} \phi^{\beta}
			& r \chi \\
		\mfg_{-1} & X
			& 0
			& -\tau_{\alpha} X^{\alpha}
			& \sqrt{2} \upsilon \eps_{\beta\alpha} X^{\beta}
			& \sqrt{2} \eps_{\alpha\beta} X^{\alpha} \phi^{\beta}
			& \chi X^{\alpha} \\
		\mfg_0    & A
			& A^{\alpha}{}_{\alpha} \sigma
			& - \tau_{\beta} A^{\beta}{}_{\alpha} + A^{\beta}{}_{\beta} \tau_{\alpha}
			& 0
			& A^{\alpha}{}_{\beta} \phi^{\beta} - A^{\beta}{}_{\beta} \phi^{\alpha}
			& -A^{\alpha}{}_{\alpha} \chi \\
		\mfg_{+1} & Z
			& - \sigma Z_{\alpha}
			& \sqrt{2} \eps^{\alpha\beta} Z_{\alpha} \tau_{\beta}
			& \sqrt{2} \upsilon \eps^{\beta \alpha} Z_{\beta}
			& Z_{\alpha} \phi^{\alpha}
			& 0 \\
		\mfg_{+2} & s
			& s \sigma
			& \frac{1}{\sqrt{2}} s \epsilon^{\beta \alpha} \tau_{\beta}
			& s \upsilon
			& 0
			& 0 \\
		\mfg_{+3} & W
			& - \sigma W^{\alpha}
			& \tau_{\alpha} W^{\alpha}
			& 0
			& 0
			& 0 \\
    \end{array} \textrm{.}
\end{equation}

\end{document}